\documentclass[a4paper]{article}

\usepackage[english]{babel}
\usepackage[utf8x]{inputenc}
\usepackage{amsmath}
\usepackage{graphicx}
\usepackage[colorinlistoftodos]{todonotes}
\usepackage{authblk}
\usepackage{fontawesome}

\usepackage{amsmath}
\usepackage{epsfig,amsthm}
\usepackage{latexsym}
\usepackage{amsfonts}
\usepackage{amssymb}
\usepackage{amscd}
\usepackage{mathrsfs}
\usepackage[all,cmtip]{xy}
\usepackage{enumerate}
\usepackage{tikz}
\usepackage{arydshln}
\usepackage{bbm}
\usepackage{mathtools}
\usepackage{leftindex}
\usepackage{mathdots}
\usepackage{centernot}
\usepackage{stmaryrd}

\usepackage{tocloft}
\usepackage{amsfonts}
\usepackage{color}
\usepackage[colorlinks]{hyperref}
\hypersetup{pdfstartview={FitH},         linkcolor=blue,  citecolor=green }

\numberwithin{equation}{section} 

\usepackage[a4paper,bindingoffset=0.2in,left=1in,right=1in,top=1in,bottom=1in,footskip=.25in]{geometry}

\usepackage{fancyhdr}
\AtEndDocument{%
  \par
  \medskip
  \begin{tabular}{@{}l@{}}%
    \textsc{Xiamen University Malaysia}
    \\
     \textsc{Jalan Sunsuria, Bandar Sunsuria, 43900 Sepang, Selangor, Malaysia}\\
    \textit{E-mail address}: \texttt{\href{mailto:kamfai.tam@xmu.edu.my}{kamfai.tam@xmu.edu.my}}
  \end{tabular}}

\newtheorem{thm}{Theorem}[section]

\newtheorem{prop}[thm]{Proposition}

\theoremstyle{definition}
\newtheorem{dfn}[thm]{Definition}
\newtheorem{rmk}[thm]{Remark}

\newtheorem{conj}[thm]{Conjecture}

\title{A conjectural construction of Arthur packets in Fargues-Scholze's categorical local Langlands correspondence}
\author{
 Geo Kam-Fai Tam }
 \affil{Xiamen University Malaysia
 \\
 {\it Email address}: \href{mailto:kamfai.tam@xmu.edu.my}{kamfai.tam@xmu.edu.my}
}
\date{\vspace{-0.5cm}\today}

\begin{document}
\maketitle


\begin{abstract}

We present a conjectural construction of Arthur packets within Fargues-Scholze's framework for the categorical local Langlands correspondence (CLLC). This construction consists of three parts. We first provide an overview the main statement of the CLLC, the underlying moduli stacks -- $\mathrm{Par}_G$ of parameters and $\mathrm{Bun}_G$ of $G$-bundles -- on the two sides of the correspondence, and its relation to representations of reductive p-adic groups. We then review the geometric Satake correspondence in order to define Hecke operators and the spectral action of sheaves on $\mathrm{Par}_G$ on sheaves on $\mathrm{Bun}_G$, and to construct semisimple parameters using excursion data. Finally, we generalize the geometric construction of Arthur packets from pushing-forward skyscraper sheaves on the regular conormal bundle of $\mathrm{Par}_G$ over $\mathbb C$ to a conjectural analogous operation on the stack of singularities on $\mathrm{Par}_G$ over $\overline{\mathbb Q}_\ell$.

\end{abstract}

{\color{red}As pointed out by Prof Teruhisa Koshikawa, Conjecture \ref{main conjecture} is ill-formulated. We will address the issue in the near future.}

\tableofcontents

\section{Introduction}

Let $G$ be a reductive group over a local field $F$, where $F$ is non-Archimedean with residue field $\mathbb F_q$ of characteristic $p>0$. For simplicity, we assume $G$ to be quasi-split over $F$. Fix $\Lambda = \mathbb C$ or $\overline{\mathbb Q}_\ell$ ($\ell\neq p$), the scalar field over which our representations and sheaves will be defined. Denote by $\hat G$ the dual group of $G$ over $\Lambda$. The main statement of the classical local Langlands correspondence (LLC) predicts a map $\pi \mapsto \varphi_\pi$ from an irreducible smooth $\Lambda$-representation $\pi$ of $G(F)$ to its Langlands parameter $\varphi = \varphi_\pi$, which is the $\hat G (\Lambda)$-conjugacy class (or sometimes called the orbit) of a continuous 1-cocycle from the Weil group $W_F$ of $F$ to $\hat G (\Lambda)$. This prediction has been verified for $G = \mathrm{GL}_n$ (e.g. \cite{Harris-Taylor, Scholze-LLC-GLn}) and for $G$ an orthogonal or symplectic group \cite{Arthur-book}. The correspondence is a bijection in the $\mathrm{GL}_n$-case and is finite-to-one in the classical group case, in which the preimage set $\Pi_\varphi$ of each parameter is called a Langlands packet, or more commonly an L-packet.

Besides its number theoretic origins, namely the Artin reciprocity in local class field theory (see, e.g., various chapters in \cite{Cassels-Frohlich-ANT}, especially the last one about Tate's thesis) and its generalization to non-abelian groups e.g., \cite{Jacquet-Langlands-GL2, BH-GL2}, motivated by equating the analytic invariants (L-functions and epsilon factors) on both sides of the correspondence, one may wonder why the LLC is set up in this way. As a geometric answer to this question, one of whose origins can be traced back to \cite{Vogan-LLC}, Fargues-Scholze \cite{FS-main} upgrades the main statement of the LLC into a derived categorical equivalence between lisse-{\'e}tale sheaves on the moduli stack $\mathrm{Bun}_G $ of $G$-bundles over the Fargues-Fontaine curve, a geometric object encoding the analytic geometry of $F$, and the coherent sheaves on another moduli stack $\mathrm{Par}_G =[Z^1(W_F,\hat G)/\hat G]$ of Langlands parameters over $\Lambda$. 
We call this upgraded statement the \emph{categorical local Langlands correspondence} (CLLC).

The above category of sheaves on $\mathrm{Bun}_G$ encodes information about the representation theory of $G(F)$. To retrieve the classical LLC from sheaves back to representations, we may  look at different `parabolic levels' of parameters. If $\varphi$ is elliptic, i.e., $\varphi(W_F)$ is not contained in any parabolics of ${} \hat G(\Lambda)$, then Fargues conjectures that representations in $\Pi_\varphi$ are given by applying the so-called \emph{spectral action} of the skyscraper sheaf of $\mathrm{Par}_G$ supported at $\varphi$ to the Whittaker sheaf on $\mathrm{Bun}_G$ (a slight upgrade of the Whittaker representation of $G(F)$). At another extreme, if the image of $\varphi$ lies in the center of $\hat G(\Lambda)$ after being twisted by the modular character of a Borel subgroup of $\hat G$, i.e., when non-trivial monodromy arises around the point in $\mathrm{Par}_G$ representing $\varphi$, then one is led to study the neighboring geometry of $\varphi$ in the nilpotent cone, which is a singular variety. This is one of the reasons why derived algebraic geometry comes in. In any case, one may still conjecture that $\Pi_\varphi$ is obtained by applying the spectral action of (the proper pushforward of) the skyscraper sheaf at $\varphi$ to the Whittaker sheaf of $G$.

Let's return to the representation theory of $G(F)$ as a topological group, and assume $\Lambda = \mathbb C$ for convenience. It is known that tempered representations of $G(F)$ behave well under the LLC when considered with the unitary dual (i.e. the representation spectrum of $L^2(G(F))$) or the theory of twisted endoscopy (as an instance of the functoriality principle). To extend the theory to non-tempered representations, Arthur introduces in \cite{Arthur-UnipotentAutomorphicRepresentations} an enlarged version of L-packets, commonly called the Arthur packets. These are again finite sets of representations, but unlike the L-packets, they may not form a partition (i.e., they may have non-empty intersections). He parameterizes these packets by the so-called Arthur parameters, which are 
smooth morphisms $\psi: W_F \times \mathrm{SL}_2(\Lambda) \times  \mathrm{SL}_2(\Lambda)\rightarrow \hat G(\Lambda)\rtimes W_F$. Eventually, he proved in his monumental colloquium \cite{Arthur-book} that representations within an Arthur packet satisfy crucial relations such as the endoscopic character identity.

One may now wonder whether Arthur packets can be constructed in Fargues-Scholze's categorical framework. The main purpose of this note is to explain such a possibility, which has been already hinted in various literature under different setups.

For example, when $F$ is an archimedean local field, \cite{ABV} has constructed a version of Arthur packets by studying the microlocal geometry on the 
conormal bundles around an orbit in $\mathrm{Par}_G$, in which case representations in an Arthur packet are parametrized (although non-bijectively) by the equivariant local systems of the corresponding microlocal fundamental group,  defined using the conormal bundle over $\mathrm{Par}_G$, at the orbit. Based on this construction, Cunningham and his team have developed in \cite{Cunningham-Voganish-begins} an analogous non-Archimedean theory for constructing the so-called ABV-packets via a particular vanishing cycle functor, {\it cf.} \cite[Expos{\'e}s XIV]{SGA7II}, relating sheaves on $\mathrm{Par}_G$ and those on its conormal bundle. When restricted to the regular part of the conormal bundle, whose orbits are parameterized by the Arthur parameters, the team shows by examples that the ABV-packets so constructed coincide with the packets originally defined by Arthur.

To connect the above microlocal geometric setup with the categorial framework, we regard the conormal bundle as a singularity stack $\mathrm{Sing}_G$ over $\mathrm{Par}_G$, and pushforward  by the adjoint of the vanishing cycle functor the skyscraper sheaf on $\mathrm{Sing}_G$ supported at the orbit corresponding to $\psi$. We call the resulting sheaf $\mathcal A_\psi$ the {\it Arthur sheaf} as in \cite{Cunningham-Voganish-begins}. At this point we see that the construction of $\mathcal A_\psi$ is completely analogous to the sheaf $\mathcal F_\varphi$ in Fargues' conjecture for elliptic Langlands parameters. Assuming the version of this conjecture for general parameters, the rest of the work to obtain the Arthur packet $\Pi_\psi$ is then straightforward: representations in $\Pi_\psi$ should be given by the spectral action of the Arthur sheaf to the Whittaker sheaf.

One feature of upgrading Arthur's theory to the above categorical framework is that: while Arthur's theory mainly focuses on understanding unitary representations of $G(F)$ (hence requiring $\Lambda = \mathbb C$), Fargues-Scholze's framework can be applied to other scalar fields such as $\Lambda =\overline{\mathbb Q}_\ell$ (or even just an algebra over $\mathbb Z_\ell$), allowing tools from other important theories, such as $\ell$-adic cohomology, to be brought into service. One may wonder whether there is a `rational' or even an `integral' version of Arthur's theory.
Another feature is that: in a sense, \cite{Cunningham-Voganish-begins} only studies the regular parts of the Arthur sheaves, although these are enough to verify the existing Arthur packets by examples. When extending from the regular conormal bundle to the whole one, we must encounter singularities around orbits of Arthur parameters, and therefore cannot avoid the corresponding derived geometry.

We now provide an outline of this note, which is divided into three parts. 
\begin{enumerate}[(1)]
\item The first part is a very rough summary of \cite[Ch III, VIII]{FS-main}. We first provide an overview of the main statement of the CLLC. We then recall the two main players: $\mathrm{Par}_G$ and $\mathrm{Bun}_G$, their definitions, and the derived categories of sheaves on these moduli stacks. We also relate the CLLC back to the classical LLC at the end.

\item In the second part, we summarize from \cite[Ch VI, VIII, XI, and X]{FS-main} some crucial facts about the geometric Satake equivalence, and use it to describe how sheaves on $\mathrm{Par}_G$ act on those on $\mathrm{Bun}_G$ by the spectral action, which is the key connection of the two sides of the CLLC. For instance, one may use the Hecke operators, combined with various excursion data, to determine the semi-simplification of the Langlands parameter of an irreducible representation of $G(F)$. Finally, we state Fargues' conjecture about L-packets corresponding to elliptic parameters, and extend it to general parameters via parabolic induction.

\item In the third part, we summarize various parts of \cite{Cunningham-Voganish-begins} by recalling first the definition of Arthur parameters and the associated Arthur packets, and then the construction of Arthur sheaves using the vanishing cycle functors from parameter stacks to regular conormal bundles in the p-adic case. We attempt to generalize these constructions to the categorical framework, and at the end formulate Conjecture \ref{main conjecture} on constructing Arthur packets categorically, assuming and closely resembling Fargues' conjecture.

\end{enumerate}
At the end of each part, we summarize some heuristic calculations on $G = \mathrm{GL}_2$ as an example of the described theories.

We emphasize that the primary purpose of this note is only to explain how Arthur's theory and the CLLC can be connected together. We hope that such a connection can benefit mathematicians from different backgrounds and allow generalizations of existing theories and further developments on current research. As a disclaimer, this note offers essentially no new theorems. All propositions and conjectures in this note can be found in the quoted references, mostly in \cite{FS-main}, \cite{FS-notes}, \cite{Fargues-overview}, \cite{Imai-notes-FS}, \cite{Hansen-Beijing-notes}, \cite{Cunningham-Voganish-begins}. Readers are strongly advised to visit these references for the original discussions and any missing details.

The author has also suppressed discussions about the relative Fargues-Fontaine curve and topics related to perfectoid spaces and condensed mathematics, which occupy some crucial parts of \cite{FS-main}, not only for the obvious reason that these subjects are highly technical, but also for the intention of relating this note to some classical geometric representation theory (e.g., \cite{Chriss-Ginzburg}) and another analogous research program known as the geometric Langlands conjecture (GLC) \cite{Arinkin-Gaitsgory-2015, Gaitsgory-GLC1}, in which they declare a similar categorical equivalence statement with the relative Fargues-Fontaine curve replaced by a smooth complete curve over a field of characteristic 0. The relative Fargues-Fontaine curve in this note will be called `the curve'. One may wonder whether the theory of CLLC is functorial with respect to this curve input.

\section{Categorical LLC}
\label{section Categorical LLC}

Let's fix the notation in this note.

Let $p \in \mathbb Z$ be a prime and  $F$ be a $p$-adic field. Put $\Gamma_F = \mathrm{Gal}(\bar F/ F)$. Denote the residue field of $F$ by $k$, which consists of $q$ elements. Let $\dot{\mathrm{Fr}}$ be a lifting of the Frobenius from $\overline{k}$ to $\breve{F}$, the completion of the maximal unramified extension of $F$.

Let $W_F$ be the Weil group of $F$, a dense subgroup of $\Gamma_F$. Most of the time, actions of $\Gamma_F$ on various objects factor through subgroups of finite index, in which cases we may replace $\Gamma_F$ by $W_F$ without any notice.

Let $G$ be a quasi-split reductive group over $F$. Fix a pinning $(T,B,\{x_\mu\}_{\mu})$, and hence a based root datum $(X^*,\Delta, X_*, \Delta^\vee)$ of $  G$. An action of $W_F$ on the dual group $\hat G$ is defined such that $W_F$ fixes the dual based root datum $(X_*, \Delta^\vee, X^*,\Delta)$.

Fix another prime  $\ell\neq p$ (which may be taken to be large enough) and a scalar field $\Lambda = \bar{\mathbb Q}_\ell$.

\subsection{An overview}

We loosely define
\begin{align*} 
\mathrm{Bun}_G &= \text{the moduli stack of $G$-bundles over the (relative Fargues-Fontaine) curve } X,
\\
\mathrm{Par}_G & = \text{the moduli stack of L-parameters (continuous 1-cocycles $\varphi:W_F\rightarrow\hat G$) of $G$},
\end{align*}
while the formal definitions will be given in Sections \ref{subsection Bun-G} and \ref{subsection Par-G} respectively. Believing these moduli stacks are important geometric objects with rich topological structures, their loose definitions are sufficient to state the main conjecture of the CLLC.

\begin{conj}
\label{FS-main-statement}
\cite[Conj X.1.4]{FS-main} (the categorical local Langlands correspondence). There is an equivalence of categories
\begin{equation}
\label{FS-main-statement-equivalence}
\mathcal D_{\mathrm{lis}}(\mathrm{Bun}_G)^\omega \xrightarrow{\quad \simeq \quad} \mathcal Coh^{b,\mathrm{qc}}(\mathrm{Par}_G)
\end{equation}
where 
\begin{itemize}
\item $\mathcal D_{\mathrm{lis}}(\mathrm{Bun}_G)$ is the derived category of lisse-{\'e}tale $\Lambda$-sheaves on $\mathrm{Bun}_G$ and $\mathcal D_{\mathrm{lis}}(\mathrm{Bun}_G)^\omega$ is the subcategory of compact objects, and

\item $\mathcal Coh^{b,\mathrm{qc}}(\mathrm{Par}_G)$ is the derived category of bounded coherent complexes of $\Lambda$-sheaves with quasi-compact support.

\end{itemize}

Moreover, there is an action of  $ \mathcal Coh^{b,\mathrm{qc}}(\mathrm{Par}_G)$ on $\mathcal D_{\mathrm{lis}}(\mathrm{Bun}_G)$, known as the spectral action, such that, by viewing $\mathcal Coh^{b,\mathrm{qc}}(\mathrm{Par}_G)$ acting on itself by the usual tensor-product, the equivalence (\ref{FS-main-statement-equivalence}) is $ \mathcal Coh^{b,\mathrm{qc}}(\mathrm{Par}_G)$-linear. 
\qed
\end{conj}

As an interesting feature, $\mathrm{Bun}_G$ and $\mathrm{Par}_G$ are moduli of objects arising from different natures (coherent v. {\'e}tale), yet their sheaf theories are closely related at the derived categorical level ({\'e}tale v. coherent). To establish the $\mathcal Coh^{b,\mathrm{qc}}(\mathrm{Par}_G)$-linear equivariance, we require machinery from the theory of Satake category, which will be covered in Section \ref{section Geometric Satake}.

\begin{rmk} The derived category $\mathcal D_{\mathrm{lis}}(\mathrm{Bun}_G)$ should be regarded as a \emph{condensed stable $\infty$-category} in the sense of Clausen-Scholze. We refer to \cite[Ch VII especially Sec VII.6, and Ch IX.1]{FS-main} for details.
A more appropriate notation is $\mathcal D_\blacksquare$, but we simply put it as $\mathcal D_{\mathrm{lis}}$ and hope that this won't cause too much confusion or any inaccuracies. 
\qed
\end{rmk}

\begin{rmk}
According to \cite[Conj 1.7.3]{Hansen-Beijing-notes}, the equivalence (\ref{FS-main-statement-equivalence}) is conjectured to be restricted from the functor 
\begin{equation}
\label{FS-main-statement-equivalenc, upgraded-to-Coh}
c_{\mathfrak f}:\mathcal D_{\mathrm{lis}}(\mathrm{Bun}_G)\xrightarrow{}\mathcal{QC}oh(\mathrm{Par}_G),
\end{equation}
where the RHS is the derived category of quasi-coherent complexes. It is also conjectured that $c_\mathfrak f$ is 
the right adjoint of the spectral action functor 
$$a_{\mathfrak f}: \mathcal{QC}oh(\mathrm{Par}_G) \to \mathcal D_{\mathrm{lis}}(\mathrm{Bun}_G), \quad \mathcal F\mapsto \mathcal F * \mathcal W_\mathfrak f,$$
where $\mathcal W_{\mathfrak f}$ is the Whittaker sheaf on $\mathrm{Bun}_G$, i.e., the proper pushforward to $\mathrm{Bun}_G$ of the Whittaker representation $ \mathrm{cInd}_{U(F)}^{G(F)}\mathfrak f$, where $U$ is the unipotent radical of $B$ and $\mathfrak f$ is a non-degenerate character of $U(F)$. In other words, knowledge the spectral action suffices to determine the equivalence (\ref{FS-main-statement-equivalence}). Occasionally we will relax the condition and allow discussions of sheaves in $\mathcal D_{\mathrm{lis}}(\mathrm{Bun}_G)$, i.e., complexes that are not necessarily compact. For instance, it is conjectured that the structure sheaf $\mathcal O_{\mathrm{Par}_{G}}$ on $\mathrm{Par}_{G}$ corresponds to $\mathcal W_{\mathfrak f} $ on $\mathrm{Bun}_G$.
\qed
\end{rmk}

\begin{rmk}
Denote by $\mathcal Perf(\mathrm{Par}_G)$ the full subcategory of perfect complexes \cite[Sec 15.76]{stacks-project} of $\mathcal Coh^{b}(\mathrm{Par}_G)$. Since our coefficient field $\Lambda$ is of characteristic 0, the situation simplifies: $\mathcal Perf(\mathrm{Par}_G)= \mathcal Coh^{b}(\mathrm{Par}_G)$. Moreover, it is known \cite[Cor
3.22]{BEN-ZVI-FRANCIS-NADLER} that the inductive limit $\mathrm{Ind} \mathcal Perf^{\mathrm{qc}}(\mathrm{Par}_G) $ is $\mathcal{QC}oh(\mathrm{Par}_G)$, which suggests the enhanced correspondence (\ref{FS-main-statement-equivalenc, upgraded-to-Coh}).
\qed\end{rmk}

\subsubsection{$\mathrm{Par}_G$}
\label{subsection Par-G}

We briefly discuss the first main player $\mathrm{Par}_G$. Let $Z^1(W_F , \hat G)$ be the (ind-)scheme representing the functor
$$\Lambda\text{-}\mathcal{A}lg\to \mathcal{S}et, \quad R \mapsto \{\text{continuous 1-cocycles }W_F\to \hat G(R)\}.$$ 
The ind-scheme structure is defined by viewing $Z^1(W_F,\hat G)$ as a union of open and closed affine subschemes $Z^1(W_F/P,\hat G)$ where $P$ ranges over all open subgroups of the wild inertia subgroup $P_F$ of $W_F$ (see \cite[Th I.8.1]{FS-main}). We put 
$$\mathrm{Par}_G := [Z^1(W_F , \hat G)/\hat G],$$ the moduli stack whose $\Lambda$-points parametrize $\hat G$-conjugacy classes of continuous 1-cocycles $W_F\to {}\hat G(\Lambda) $.

We call an L-parameter {\bf semisimple} if whenever it factors over a relevant (in the sense of \cite[Sec 3]{Borel-L-functions} and relative to the chosen root datum of $G$) parabolic subgroup $\mathcal P$ of $\hat G(\Lambda)$, it also factors over a Levi in $\mathcal P$. These are precisely closed $\hat G$-conjugacy classes in $Z^1(W_F , \hat G)$.

We denote by $\mathrm{par}_G$ the {\bf coarse quotient space} of $\mathrm{Par}_G$, i.e., the GIT-quotient $Z^1(W_F , \hat G)\sslash \hat G$. In other words, the structure ring of $\mathrm{par}_G$ is $\mathcal O(Z^1(W_F , \hat G))^{\hat G}$. By \cite[Prop VIII.3.8]{FS-main}, the variety of closed points $|\mathrm{par}_G|$ of $\mathrm{par}_G$ (i.e., geometric points of
$\mathrm{Spec} (\mathcal O(Z^1(W_F , \hat G))^{\hat G})$, which are the same as points in $|\mathrm{Par}_G|$) parametrizes semisimple L-parameters.

\begin{rmk}
Since $W_F/P$ is finitely generated as an abstract group, we see that $Z^1(W_F/P,\hat G)$ is finite dimensional, and hence $\mathrm{Par}_G $ is a disjoint union of finite type algebraic stacks over $\Lambda$. 
Another crucial property  \cite[Def VIII.1.3]{FS-main} is that $Z^1(W_F/P,\hat G)$ (as well as $Z^1(W_F,\hat G)$) is a flat and local complete intersection (i.e., syntomic \cite[Def VIII.2.6]{FS-main}). These properties are important for studying the singularity of $\mathrm{Par}_G$ in Section \ref{subsection Conjectural Arthur packet in CLLC}.
\qed\end{rmk}

\subsubsection{$\mathrm{Bun}_G$}
\label{subsection Bun-G}

We now discuss the second main player $\mathrm{Bun}_G$. Hidden in the background are $\mathcal Perf_k$, the category of perfectoid spaces over $k$, and $X$, the functor of the curve on $\mathcal Perf_k$.
The functor $\mathrm{Bun}_G$ is the (small v-pre)stack taking each $ S\in \mathcal{P}erf_k$ to the groupoid
of (finite rank) $G$-bundles on $X_S$: 
$$\mathrm{Bun}_G(S) = \{ \mathcal E: \text{$G$-bundles over $X_S$}\},$$
There is a stratification of $\mathrm{Bun}_G$ as 
\begin{equation}
\label{stratification BunG}
\mathrm{Bun}_G = \bigsqcup_{b\in B(G)} \mathrm{Bun}^b_G ,
\end{equation}
where $B(G)$ is the Kottwitz set of $G$-isocrystals, i.e., the set of Frobenius-conjugacy classes in $G(\breve{F})$. Each substack embedding
$$i_b: \mathrm{Bun}^b_G \hookrightarrow \mathrm{Bun}_G $$ is defined by identifying $B(G)$ with the closed-point subset $|\mathrm{Bun}_G|$ as follows. The set $B(G)$ is equipped with two maps, 
\begin{align}
&\text{the Kottwitz invariant $\kappa: B(G)\to \pi_1(G)_{\Gamma_F}$, \quad  and }
\\
&\text{the Newton/Harder-Narasimhan polygon 
$\nu: B(G)\rightarrow (X_*(T)^+_{\mathbb Q} )^{\Gamma_F}$.}
\end{align}
Here $X_*(T)_{\mathbb Q}^+$ is the poset of dominant cocharacters of $G$ with rational coefficients. We order $B(G)$ by putting $b\leq b'$ if and only if $\kappa(b) = \kappa(b')$ and $\nu(b)\leq \nu(b')$. In this way, each stratum in (\ref{stratification BunG}) contains a closed point of $\mathrm{Bun}_G$, and therefore $|\mathrm{Bun}_G|\rightarrow B(G)$ is continuous if $B(G)$ is equipped with the order topology:
$$b\leq b' 
\quad\Leftrightarrow\quad b\in \overline{\{b'\}},$$
 and is indeed a homeomorphism by \cite{Viehmann}.

A $G$-bundle $\mathcal E$ is  {semi-stable} if it has no proper non-zero subbundles with polygon $>\nu(\mathcal E)$.  Let $\mathcal E_b$ be the canonical $G$-bundle defined by the Frobenius action twisted by $b\in B(G)$. We call $b$ {\bf basic} if $\mathcal E_b$ is semi-stable. Denote by $B(G)_{\mathrm{bas}} $ the subset of $B(G)$ consisting of basic $G$-isocrystals.

As a form of Kottwitz duality, we have bijections \cite[Th I.4.1 and Cor IV.1.23]{FS-main}
$$B(G)_{\mathrm{bas}} \xrightarrow{\kappa}\pi_1(G)_{\Gamma_F} \rightarrow \pi_0(\mathrm{Bun}_G),$$ 
i.e., basic isocrystals parametrize the connected components of $\mathrm{Bun}_G$. Each component $\mathrm{Bun}_G^b$, for $b\in B(G)_{\mathrm{bas}}$, contains a unique semi-stable point (corresponding to $\mathcal E_b$), and defines an open immersion into $\mathrm{Bun}_G$. 

For example, we now look at $\mathrm{Bun}_G$ when $G = \mathrm{GL}_n$ where $n=1$ or $2$. For $\mathrm{GL}_1$, we have $B(G) = B(G)_{\mathrm{bas}}= \mathbb Z$, and  
\begin{equation}
\label{Bun-GL1}
\mathrm{Bun}_{\mathrm{GL}_1}^{}= \mathrm{Bun}_{\mathrm{GL}_1}^{\mathrm{ss}}= \bigsqcup_{\mathbb Z}[*/\underline{F^\times}],
\end{equation}
where $*$ is a shorthand notation for a point, over $S$ in this context, and each stack $[*/\underline{F^\times}]$ represents the classifying space of degree $n$ line bundles on the curve $X$. (If $\mathcal X$ is a topological space, we define $$\underline{\mathcal X}: \mathcal Perf_k \to \mathcal Set, \quad  S \mapsto \mathcal C(|S|, \mathcal X) \text{ the set of continuous maps $|S|\to \mathcal X$},$$  
the corresponding constant sheaf functor.)

For $G=\mathrm{GL}_2$, whose inner forms are parametrized by $H^1(W_F, \mathrm{PGL}_2)$, we use the Kottwitz duality 
\begin{equation}
\label{Kottwitz duality}
\pi_1(G_{\mathrm{ad}})_{\Gamma_F}  \cong X^*((Z\hat G_{\mathrm{sc}})^{\Gamma_F})
\end{equation}
 to obtain $B(G)_{\mathrm{bas}} = \mathbb Z/2$. Here the non-trivial basic isocrystal is represented by anti-diag$(\varpi,1)\in G(\breve{F})$. By considering polygons, we identify $\mathbb Z/2$ with $\{0,\tfrac{1}{2}\}$, and put $G_0(F)=\mathrm{GL}_2(F)$ and $G_{1/2} = \mathrm{D}_2(F)^\times$, where $\mathrm{D}_2$ is the quaternion algebra over $F$. The stratification of $\mathrm{Bun}^{\mathrm{ss}}_G$ is given by 
$$ \mathrm{Bun}_{\mathrm{GL}_2}^{\mathrm{ss}}= \bigsqcup_{\mathbb Z}[*/\underline{\mathrm{GL}_2(F)}] \sqcup \bigsqcup_{\tfrac{1}{2}\mathbb Z \smallsetminus \mathbb Z}[*/\underline{D^\times_2(F)}],$$
where each point represents rank-2 vector bundles on $X$ with a fixed Harder-Narasimhan component. We refer to \cite{Nyguen-GLn} for the calculations involving other $\mathrm{GL}_n$.

In general, $ \pi_0(|\mathrm{Aut}(\mathcal E_b)|)$ is equal to the centralizer $G_b$ of $b$ in $G$, i.e.,  an inner form of a Levi subgroup of (the quasi-split inner form of) $G$. If $b\in B(G)_{\mathrm{bas}}$, then $G_b$ is called an {\bf extended pure inner form} of $G$. By \cite[Th VII.7.1]{FS-main}, the natural map
$$\mathrm{Bun}^b_
G = [* / \mathrm{Aut}(\mathcal E_b)] \to [*/\underline{G_b(F)}]$$
induces an equivalence 
$$\mathcal D_{\mathrm{lis}}(\mathrm{Bun}^b_
G)\cong \mathcal  Rep(G_b(F)),$$
where the RHS is the derived category of the abelian category of smooth $G_b(F)$-representations.

\begin{rmk}
By \cite[Th VII.7.3]{FS-main}, there is a semi-orthogonal decomposition of $\mathcal D_{\mathrm{lis}}(\mathrm{Bun}_G)$ into $\mathcal D_{\mathrm{lis}}(\mathrm{Bun}^b_
G)$ with $b$ ranging in $B(G)$, where  semi-orthogonality means that for any complexes $\mathcal F \in \mathcal D_{\mathrm{lis}}(\mathrm{Bun}^b_
G)$ and $\mathcal G \in \mathcal D_{\mathrm{lis}}(\mathrm{Bun}^{b'}_
G)$, we have $\mathrm{Hom}(i_{b!}\mathcal F, i_{b'!}\mathcal G)\neq 0$ only when $b\leq b'$. 
\qed\end{rmk}

\subsection{Relation to the classical LLC}

To see how the categorical conjecture  relates to the classical conjecture about irreducible representations of $G(F)$, we provide another version of Conjecture \ref{FS-main-statement} under the condition that $\varphi$ is elliptic (i.e., ${(Z\hat G)^{ W_F}}$ is of finite index in the stabilizer group $S_\varphi:=Z_{\hat G}(\varphi)$) in Conjecture \ref{FS-main-statement-classical} below.

Given a $\Lambda$-point $x_\varphi$ in $\mathrm{Par}_G$ representing an elliptic  L-parameter $\varphi$, let $C_\varphi$ be the connected component containing $x_\varphi$, whose points correspond to unramified twists of $\varphi$: 
$$\varphi\mapsto \varphi(\chi\circ \nu), \quad \chi\in X_*(((Z\hat G)^{W_F})^\circ),$$
where $\nu:W_F\to \Lambda^\times$ is the unramified character with $\nu(\dot{\mathrm{Fr}})=q^{-1}$. There exists a subcategory  $\mathcal D^{C_\varphi}_{\mathrm{lis}}(\mathrm{Bun}_G)^\omega$ of $\mathcal D^{}_{\mathrm{lis}}(\mathrm{Bun}_G)^\omega$ associated to $C_\varphi$, which is a direct summand of $\mathcal D^{}_{\mathrm{lis}}(\mathrm{Bun}_G)^\omega$ and such that, informally speaking, every Schur-irreducible object therein has its L-parameter in $C_\varphi$.

\begin{conj}
\label{FS-main-statement-classical}
\cite[Conj X.2.2]{FS-main}
Let $\varphi$ be an elliptic L-parameter. There exists a unique generic representation $\pi_\varphi$ of $G(F)$ with L-parameter $\varphi$, and an equivalence of categories
$$\mathcal Perf([*/S_\varphi])  =  \mathcal Rep_{\mathrm{fin}\text{-}\mathrm{proj}}(S_\varphi) \rightarrow 
\mathcal D^{C_\varphi}_{\mathrm{lis}}(\mathrm{Bun}_G)^\omega, \quad V \mapsto V*\pi_\varphi, $$
such that each $\rho \in  \mathcal Rep_{}(S_\varphi) ^{\mathrm{sim}}_{/\mathrm{iso}}$ (the set of isomorphism classes of simple objects) corresponds to a(n irreducible supercuspidal) representation $\pi_{\varphi,\rho}$ of $G_b(F)$, where $G_b$ is the inner form of $G$ defined by the $G$-isocrystal $b$ corresponding to $ \rho|_{(Z\hat G)^{W_F}}$ via Kottwitz duality (\ref{Kottwitz duality}).\qed\end{conj}

The Fargues-Scholze L-packet $\Pi_\varphi^{FS}$ then consists of $\pi_{\varphi,\rho}$ where $\rho$ ranges over $ \mathcal Rep_{}(S_\varphi) ^{\mathrm{sim}}_{/\mathrm{iso}}$. Each $\pi_{\varphi,\rho}$ is the representation conjecturally corresponding to $(\varphi,\rho)$ under the classical LLC. A more precise description of  $\pi_{\varphi,\rho}$, proposed by Fargues, will be given in Conjecture \ref{Fargues-conjecture-detailed}  after we develop a deeper understanding of the spectral action.

\begin{rmk} We describe a reduction of the action of $\mathcal Perf([*/S_\varphi])$ on $\mathcal D^{C_\varphi}_{\mathrm{lis}}(\mathrm{Bun}_G)^\omega$ in Conjecture \ref{FS-main-statement-classical}. Note that the full subcategory $\mathcal Perf([*/ ({S_\varphi}/(Z\hat G)^{W_F} )])$ of $\mathcal Perf([*/S_\varphi])$, consisting of those representations of $S_\varphi$ on which $(Z\hat G)^{W_F}$ acts trivially, acts on $\mathcal D^{C_\varphi}(\mathrm{Bun}^b_G)^\omega$, \emph{cf.} \cite[Sec X.2]{FS-main}.  Denote by $\mathcal D(G_b(F))_{\mathrm{fin},\varphi}$ the full subcategory of $\mathcal D^{C_\varphi}(\mathrm{Bun}^b_G)^\omega$ consisting of $\mathcal F\in \mathcal  Rep(G_b(F))$ such that
\begin{itemize}
\item $H^\bullet(\mathcal F):=\oplus_m H^m(\mathcal F)$ is of finite length; 

\item all  irreducible subquotients of $H^\bullet(\mathcal F)$ belong
to $\Pi_\varphi^{FS}$.

\end{itemize}
By  \cite[Cor 3.7(b)]{Kazhdan-Varshavsky}, the action $\mathcal Perf([*/ ({S_\varphi}/(Z\hat G)^{W_F} )])$ on $\mathcal D^{C_\varphi}(\mathrm{Bun}^b_G)^\omega$ preserves $\mathcal D(G_b(F))_{\mathrm{fin},\varphi}$.
\qed\end{rmk}

\section{Fargues-Scholze's construction}

We now turn to the technical aspects of the construction of the categorical equivalence (\ref{FS-main-statement-equivalence}) in the main Conjecture \ref{FS-main-statement}. In one direction, we employ the theory of the geometric Satake correspondence to construct sheaves of representations of $G(F)$ by applying the spectral actions of sheaves over the parameter stack $\mathrm{Par}_G$ on the Whittaker sheaf of $G$. In the opposite direction, we use excursion data to construct a semisimple Langlands parameter of a given irreducible representation of $G(F)$.

\subsection{Geometric Satake correspondence}
\label{section Geometric Satake}

Here we recall very briefly the formal definition of a Satake category. For motivation, we first consider its `degree-1 divisor' version, which is analogous to the construction in the literature \cite[Lec 19]{Scholze-Weinstein}, \cite{Zhu-affine-Grass}.

Let $B^+_{\mathrm{dR}}:\mathcal{ P}erf_{/k}\to \mathcal Ring_{top}$ be the functor of de Rham period ring \cite[Def 12.4.3]{Scholze-Weinstein}, i.e., roughly speaking, the complete local ring of the  curve at a point, and $B_{\mathrm{dR}}$ denote its field of fractions. We define the {\bf loop group} $LG$ and the {\bf jet group} $L^+G$, respectively, by 
\begin{equation}
\label{loop and jet groups, singleton divisor}
 LG (S) := G(B_{\mathrm{dR}}(S)) \hookleftarrow  L^+G (S) := G(B^+_{\mathrm{dR}}(S)) .
\end{equation}
The quotient stack $Gr_G := [ LG/L^+G ]$, called the {\bf Beilinson-Drinfeld Grassmannian}, can be realized as the moduli stack 
\begin{equation}
\label{BeilinsonDrinfeld Grassmannian}
Gr_G  = \{
(\mathcal E,\beta):\begin{smallmatrix}
\text{$\mathcal E$ is a $G$-bundle over $B_{\mathrm{dR}}^+$,}
\\
\text{$\beta$ is a trivialization of $\mathcal E$ over $B_{\mathrm{dR}}$}
   \end{smallmatrix} \}\xrightarrow{(\mathcal E,\beta)\mapsto \mathcal E} \mathrm{Bun}_G .
\end{equation}
Moreover, we define the double quotient stack, also known as the {\bf local Hecke stack} and realized as 
 $$\mathcal Hck_G :=[L^+G\backslash LG/L^+G ] =  \{
(\mathcal E_1,\mathcal E_2,\beta):\begin{smallmatrix}
\text{$\mathcal E_1,\mathcal E_2$ are $G$-bundles over $B_{\mathrm{dR}}^+$,}
\\
\text{$\beta:\mathcal E_1\xrightarrow{\simeq} \mathcal E_2$ over $B_{\mathrm{dR}}$}
   \end{smallmatrix} \}.$$
Note that the stratification of $\mathrm{Bun}_G$ using the HN-polygons, or equivalently, rational dominant cocharacters of $G$, induces a corresponding stratification on $\mathcal Hck_G $ via (\ref{BeilinsonDrinfeld Grassmannian}) and the canonical projection $Gr_G \to \mathcal Hck_G $.

Let 
$\mathcal Sat_G := \mathcal Sat(\mathcal  Hck_{G})$ be the full subcategory of \emph{universally locally acyclic (ULA)} and \emph{flat-perverse} objects of $\mathcal D_{\text{{\'e}t}}(\mathcal  Hck_{G} )$, called the {\bf Satake category}. It is a symmetric monoidal category via Lusztig's convolution product \cite[Sec VI.8]{FS-main}. The hypercohomology functor
$$H^\bullet: \mathcal  Sat_G \rightarrow \Lambda\text{-}\mathcal Vect$$
is then a symmetric monoidal functor, turning $\mathcal Sat_G$ into a (neutral) Tannakian category. The general theory of Tannakian categories determines the structure of the automorphism group of this functor.

\begin{prop}
\label{Tannakian group Langlands dual} \cite[Sec 2.5]{Zhu-affine-Grass}.
The (Tannakian) group $ \mathrm{Aut} H^\bullet$  is isomorphic to the Langlands dual group $\hat G$. 
\qed\end{prop}

Put ${}^LG = \hat G\rtimes W_F $, commonly known as the L-group of $G$, where the $W_F$-action is defined by dualizing a chosen $W_F$-invariant pinning of $G$ (up to an explicit cyclotomic twist, see the discussion around \cite[Th VI.11.1]{FS-main}). We omit the details, but remark that a pinning is already chosen at the beginning of Section \ref{section Categorical LLC}.

We can now state the main result of the geometric Satake equivalence.
\begin{prop}
\label{MVG theorem of geometric Satake}
\cite{Mirkovic-Vilonen} (with a fixed choice of $\sqrt q\in \Lambda$). There is an
equivalence of symmetric monoidal categories
 $$\mathfrak S: \mathcal Rep_{\mathrm{fin\text{-}proj}}({}^LG)\rightarrow \mathcal Sat_G,$$
from the derived category of finite dimensional projective continuous $\Lambda$-representations of ${}^LG$ to the Satake category of $G$, such that $r_\mu:= \mathrm{Ind}^{{}^LG}_{\hat G\rtimes W_\mu}V_\mu$, where $V_\mu$ is the highest weight representation of $\hat G$ corresponding to a cocharacter $\mu$ of $G$ and $W_\mu$ is the stabilizer of $\mu$ in $W_F$, corresponds to $IC_\mu$, the canonical IC-sheaf on the stratum $\mathcal Hck_{\leq \mu}$.
\qed\end{prop}

\subsubsection{Drinfeld's Lemma}

We will extend the above setup from a point on the curve $X$ to an effective divisor on $X$ of an arbitrary degree. Before doing so, let's consider a crude form of the CLLC in (\ref{crude CLLC}) below, also known as Drinfeld's Lemma.
 
Let $\mathrm{Div}^d_X$ be the moduli stack of degree $d$ effective Cartier divisors over $\mathcal Perf_k$ on $X$. We first consider the case $d=1$. If we express $\mathrm{Div}^1_X$ as $(\mathrm{Spa} \breve F)^\diamond /\dot{\mathrm{Fr}}^{\mathbb Z}$ where $\mathrm{Spa} \breve F$ is the analytic adic-spectrum and $\diamond$ is the diamond functor (with details given in \cite{Scholze-Weinstein}), then we have a canonical map
$$\mathrm{Div}^1_X = [(\mathrm{Spa} \breve F)^\diamond/\dot{\mathrm{Fr}}^{\mathbb Z}] \cong [(\mathrm{Spa} \widehat{\overline{F}})^\diamond/\underline{W_F}]\to [*/\underline{W_F}],$$
which induces an equivalence of categories \cite[Prop's IV.7.3 and VI.9.2]{FS-main},
\begin{equation}
\label{crude CLLC}
\mathcal  Rep_{\mathrm{fin\text{-}proj}}(W_F) \to \mathcal{LS}(\mathrm{Div}^1_X) ,
\end{equation}
where the LHS is the derived category of continuous representations of $W_F$ on finite projective $\Lambda$-modules, and the RHS is the derived category of $\Lambda$-local systems on $\mathrm{Div}^1_X$.

Attentive readers may notice that (\ref{crude CLLC}) is closely related to the classical Artin reciprocity. This follows directly from the line-bundle map 
$$\mathrm{Div}^1_X \xrightarrow{D \mapsto \mathcal O_X(D)} \mathrm{Pic}_X = 
\mathrm{Bun}_{\mathrm{GL}_1} = \bigsqcup_{\mathbb Z}[*/\underline{F^\times}], $$ 
where the equality on the far RHS is given by (\ref{Bun-GL1}).

\subsubsection{(Slightly generalized) geometric Satake}
\label{section Slightly generalized Geometric Satake}

We now explain a relative version of the geometric Satake equivalence over $\mathrm{Div}^d_X$. Fix a closed Cartier
divisor $D: \mathcal Perf_{/k}\to \mathrm{Div}_X^d$  of degree $d$ on $X$, and let $\mathcal I$ be its invertible ideal sheaf. We put
\begin{equation}
   \begin{split}
     B^+_{\mathrm{Div}^d_X }(S) &=\text{the global sections of the completion of }\mathcal O_{X_S}\text{ along }\mathcal I\text{, and }
     \\
     B_{\mathrm{Div}^d_X }(S) &= B^+_{\mathrm{Div}^d_X }(S)[\mathcal I^{-1}].
   \end{split}
\end{equation}
We therefore (re)define the loop and jet groups  $L_{\mathrm{Div}^d_
X }G$ and $L_{\mathrm{Div}^d_
X }^+G$ by replacing $(B_{\mathrm{dR}}(S),B_{\mathrm{dR}}^+(S))$ in (\ref{loop and jet groups, singleton divisor}) with  $(  B_{\mathrm{Div}^d_X }(S),  B^+_{\mathrm{Div}^d_X }(S))$ (the latter reduces to the former with $d=1$), and also 
\begin{equation*}
   \begin{split}
     \mathcal Hck_{G, \mathrm{Div}^d_
X } & :=[L_{\mathrm{Div}^d_
X }^+G\backslash L_{\mathrm{Div}^d_
X }G/L_{\mathrm{Div}^d_
X }^+G ] \text{ and}
\\
   \mathcal Sat_{G, \mathrm{Div}^d_
X }  & :=   \mathcal Sat(\mathcal Hck_{G, \mathrm{Div}^d_
X }  ).
   \end{split}
\end{equation*}
Let $I$ be an index set with $|I|=d$. Using the summation morphism, which sends ordered tuples to multisets,
\begin{equation}
\label{unordering morphism on DivX}
(\mathrm{Div}^1_
X)^{I} \rightarrow \mathrm{Div}^d_
X = (\mathrm{Div}^1_
X)^{I} /\Sigma_d, 
\end{equation}
where $\Sigma_d$ is the symmetric group on $d$ elements, we define the local Hecke stack and the Satake category,
\begin{equation*}
   \begin{split}
     \mathcal Hck_G^I := \mathcal Hck_{G, \mathrm{Div}^d_
X } \times_{\mathrm{Div}^d_
X } (\mathrm{Div}^1_
X)^{I} \quad \text{and}\quad 
   \mathcal Sat_G^I  :=   \mathcal Sat(\mathcal Hck_G^I ).
   \end{split}
\end{equation*}
Proposition \ref{MVG theorem of geometric Satake} is therefore upgraded into an equivalence
\begin{equation}
\label{upgraded geometric Satake} \mathfrak S^I: \mathcal Rep_{\mathrm{fin\text{-}proj}}(({}^LG)^I)\rightarrow \mathcal Sat^I_G
\end{equation}
 of symmetric monoidal categories. We also obtain a crude form of the CLLC,
\begin{equation}
\label{crude CLLC, I-version}
\mathcal  Rep_{\mathrm{fin\text{-}proj}}((W_F)^I) \to \mathcal{LS}((\mathrm{Div}^1_X)^I),
\end{equation}
which slightly extends (\ref{crude CLLC}) to divisors of multi-degree on $X$.

\subsubsection{Hecke operators}
\label{subsection Hecke operators}

The following result is crucial to defining the functor (\ref{the general functor connecting Galois rep and Bernstein center}), which will be used to construct semisimple parameters from irreducible representations of $G(F)$ in Section \ref{subsubsection The construction of semisimple parameters}.

Define the {\bf global Hecke stack}:
$$\mathrm{Hck}^ I_G (S) := \{
(\mathcal E_1,\mathcal E_2,D,\beta):\begin{smallmatrix}
\text{$\mathcal E_1,\mathcal E_2$ are $G$-bundles over $X_S$,}
\\
\text{$D:S\to (\mathrm{Div}^1_
X)^{I}$ a closed Cartier divisor of $X_S$,}
\\
\text{$\beta:\mathcal E_1\xrightarrow{\simeq} \mathcal E_2$ on $X_S\smallsetminus D$}
   \end{smallmatrix} \},$$
 with the canonical global-to-local restriction $ \mathrm{Hck}^ I_G\rightarrow \mathcal Hck^ I_G$ defined by forgetting $D$ via the map (\ref{unordering morphism on DivX}). The geometric Satake transform $\mathfrak S^I$ (\ref{upgraded geometric Satake}) induces the following diagram:
$$\mathcal Rep_{\mathrm{fin\text{-}proj}}(({}^LG)^I)\xrightarrow{\mathfrak S^I} \mathcal Sat^I_G \hookrightarrow \mathcal D( \mathcal Hck^ I_G) \rightarrow \mathcal D( \mathrm{Hck}^ I_G), \quad V\mapsto S_V.$$
Before constructing the functor in (\ref{the general functor connecting Galois rep and Bernstein center}), let's recall an operation from \cite[Prop VII.3.1]{FS-main}. The geometric objects above are defined over $\mathcal Perf_k$. They are viewed as \emph{small v-stacks} by equipping them with the topology given by v-coverings. Given a morphism $f:X\to Y$ of small v-stacks, the pull-back $f^*$ admits a left adjoint 
$$f_{\natural}: \mathcal D(X,\Lambda)\to \mathcal D (Y,\Lambda),$$
which is sometimes called the \emph{hadal pushforward}; see \cite{Hansen-Beijing-notes} for a detailed discussion of  this functor and its significance.

Put $p_1:(\mathcal E_1,\mathcal E_2,D,\beta)\mapsto \mathcal E_1$ and $p_2:(\mathcal E_1,\mathcal E_2,D,\beta)\mapsto (\mathcal E_2,D)$.  
Via the diagram $\mathrm{Bun}_G \xleftarrow{p_1} \mathrm{Hck}^I_G \xrightarrow{p_2}  \mathrm{Bun}_G \times (\mathrm{Div}^1_
X )^I $, the functor
$$\mathcal D_{\mathrm{lis}}(\mathrm{Bun}_G)\rightarrow \mathcal D(\mathrm{Bun}_G \times (\mathrm{Div}^1
_X )^I),\quad A\mapsto p_{2\natural}(p_1^*A\otimes^{} S_V )
$$
for a fixed $V\in \mathcal Rep_{\mathrm{fin\text{-}proj}}(({}^LG)^I)$ induces, via (\ref{crude CLLC, I-version}), a Hecke operator
$$T_V: \mathcal D_{\mathrm{lis}}(\mathrm{Bun}_G) \to \mathcal D_{\mathrm{lis}}(\mathrm{Bun}_G\times [*/(W_F)^I]) = \mathcal D_{\mathrm{lis}}(\mathrm{Bun}_G)^{ B(W_F)^I}.$$
This operator is crucial to defining Hecke eigensheaves in Fargues' conjecture. Note that the functor
\begin{equation}
\label{the general functor connecting Galois rep and Bernstein center}
T:\mathcal Rep_{\mathrm{fin\text{-}proj}}(({}^LG)^I) \to \mathrm{End}( \mathcal D_{\mathrm{lis}}(\mathrm{Bun}_G)), \quad V \mapsto T_V
\end{equation}
is $\mathcal Rep_{\mathrm{fin\text{-}proj}}((W_F)^I)$-linear and symmetric monoidal, i.e., $T_{V\otimes W} = T_V\circ T_W = T_W\circ T_V  $.

\subsubsection{The spectral action}

Finally, using the functor of Hecke operators, we define the spectral action of $ \mathcal Coh^{b,\mathrm{qc}}(\mathrm{Par}_G)$ on $\mathcal D_{\mathrm{lis}}(\mathrm{Bun}_G)$, as mentioned in the main Conjecture \ref{FS-main-statement}. 

There is a $\mathcal Rep_{\mathrm{fin\text{-}proj}}((W_F)^I)$-linear symmetric monoidal functor, called the universal $({}^LG)^I$-torsor,
\begin{equation}
\label{universal LG-torsor}
\mathcal Rep_{\mathrm{fin\text{-}proj}}(({}^LG)^I) \to \mathcal Perf(\mathrm{Par}_G),\quad V\mapsto (\mathcal O_{Z^1(W_F,\hat G)}\otimes_\Lambda V)/\hat G(\Lambda).
\end{equation}
\cite[Th X.1.1]{FS-main} implies that the spectral action is determined by requiring the following composition
$$\mathcal Rep_{\mathrm{fin\text{-}proj}}(({}^LG)^I) \xrightarrow{\text{(\ref{universal LG-torsor})}} \mathcal Perf(\mathrm{Par}_G)^{(W_F)^I}\dashrightarrow\mathrm{End}( \mathcal D_{\mathrm{lis}}(\mathrm{Bun}_G)^\omega)^{ B(W_F)^I}$$
to be equal to (\ref{the general functor connecting Galois rep and Bernstein center}). We describe the action as 
$$ \mathcal Perf(\mathrm{Par}_G)^{(W_F)^I}\to\mathrm{End}( \mathcal D_{\mathrm{lis}}(\mathrm{Bun}_G)^\omega), \quad \mathcal F\mapsto (A\mapsto \mathcal F * A).$$
By taking limit on (\ref{universal LG-torsor}) to $\mathrm{Ind} \mathcal Perf(\mathrm{Par}_G)^{(W_F)^I} = \mathcal {QC}oh(\mathrm{Par}_G)$, we see that the spectral action is represented by a universal object in $\mathcal D_{\mathrm{lis}}(\mathrm{Bun}_G)$ corresponding to $\mathcal O_{Z^1(W_F,\hat G)}/\hat G(\Lambda) = \mathcal O_{\mathrm{Par}_G}$, which is conjecturally the Whittaker sheaf $\mathcal W_{\mathfrak f}$ \cite[Conj X.1.4]{FS-main}.

\subsection{The excursion algebra}

Following the recipe in \cite[Sec IX.4]{FS-main}, we construct, for each (Schur-)irreducible object $A\in \mathcal D_{\mathrm{lis}}(\mathrm{Bun}_G)$, a unique
semisimple L-parameter using the  excursion operators defined by V. Lafforgue. To explain this construction, we first recall the three related Bernstein centers.

\subsubsection{The representation-theoretic Bernstein center}

For any open pro-$p$ subgroup $K$ of $G(F)$, define $\mathcal H(G(F),K) := \mathrm{End}_{G(F)}(\mathrm{cInd}^{G(F)}
_K \mathbf 1_K)$, the classical Hecke algebra of level $K$. Denote its center by $\mathcal Z(G(F),K)$. We then define the (representation-theoretic) Bernstein center of $G(F)$ by
$$\mathcal Z( G(F)) = \lim_{
   \begin{smallmatrix} 
   \longleftarrow
   \\
   K\subset G(F)
   \\
   \text{open pro-$p$}
      \end{smallmatrix}
}\mathcal Z(G(F),K).$$
This is well known to be isomorphic to the categorical center, i.e., the algebra of endo-transforms of the identity functor, of the category of smooth $\Lambda$-representations of $G(F)$.

\subsubsection{The geometric Bernstein center}

This is the categorical center of $\mathcal D_{\mathrm{lis}}(\mathrm{Bun}_G)$, i.e., 
$$\mathcal Z^{\mathrm{geom}}( G) :=\pi_0 \mathrm{End}(\mathrm{id}_{\mathcal D_{\mathrm{lis}}(\mathrm{Bun}_G)}).$$
The fully faithful functor $i_{1!}:\mathcal D(G(F)) \hookrightarrow \mathcal D_{\mathrm{lis}}(\mathrm{Bun}_G)$ induces a morphism of $\Lambda$-algebras,
$$\mathcal Z^{\mathrm{geom}}( G)\to \mathcal Z( G(F)).$$

\subsubsection{The spectral Bernstein center}

This is the ring of globally defined functions on $|\mathrm{Par}_G| = Z^1(W_F,\hat G)//\hat G$, i.e., 
$$\mathcal Z^{\mathrm{spec}}(G) = \mathcal O (Z^1(W_F,\hat G))^{\hat G}.$$
Following \cite[Th VIII.3.2]{FS-main}, we recall the algebra of excursion operators, which reduces the ring $\mathcal O (Z^1(W_F,\hat G))^{\hat G}$ to a limit of similar rings, but defined by data of finite nature.

Fix an open subgroup $P$ of $P_F$ and a subgroup $W = W(P)\subseteq W_F$ such that $W/P$ is discrete and dense in $W_F/P$. Define the algebra of {\bf excursion operators} by
$$\mathrm{Exc}(W, \hat G) :=
{\begin{matrix} 
\,\
\\
 \mathrm{colim}
 \\
   \begin{smallmatrix} 
 \\
 n,\, F_n
 \\
    F_n \to W
        \end{smallmatrix}
        \end{matrix}}
\mathcal O(Z^1(F_n, \hat G))^{\hat G}$$
where the colimit runs over all morphisms from a free group $F_n$ on $n$ generators, for all $n$, into $W$. (It may be viewed as a discretization of $\mathrm{Par}_G$ in some sense.) If $\ell$ is a good prime, then there is an isomorphism \cite[Th VIII.3.6]{FS-main}:
\begin{align*} 
\mathrm{Exc}(W, \hat G) 
\xrightarrow{\cong }\mathcal O(Z^1(W/P, \hat G))^{\hat G} \cong  \mathcal O(Z^1(W_F/P, \hat G))^{\hat G}. 
\end{align*}
(The last isomorphism requires that $\ell\in \Lambda^{\times}$, which is automatic in our case.) This induces  a canonical map 
\begin{equation}
\label{Excursion-W to Z-spec}
\mathrm{Exc}(W, \hat G)\to \mathcal Z^{\mathrm{spec}}(G).
\end{equation}
Denote by $\mathcal C_P:={\mathcal {D}^P_{\mathrm{lis}}(\mathrm{Bun}_G)}^\omega$ the subcategory of objects $A\in \mathcal D_{\mathrm{lis}}(\mathrm{Bun}_G)^\omega$ for which the action of $P$ on $T_V(A)$ is trivial for all $V\in \mathcal Rep({}^LG)$. Analogously to (\ref{Excursion-W to Z-spec}), there is a morphism \cite[Th VIII.4.1]{FS-main}:
\begin{equation}
\label{Excursion-W to Z-geom-P}
\mathrm{Exc}(W, \hat G)\to \mathcal Z({\mathcal {D}^P_{\mathrm{lis}}(\mathrm{Bun}_G)}^\omega) = \pi_0 \mathrm{End}(\mathrm {id}_{\mathcal C_P} ),
\end{equation}
which will be explained in the next subsection. Using the decomposition into connected components (i.e., letting $P$ vary), we obtain a morphism relating the geometric and spectral Bernstein centers,
$$\mathcal Z^{\mathrm{spec}}( G)\to \mathcal Z^{\mathrm{geom}}( G),$$
which, by construction, factors over the subalgebra $\mathcal Z^{\mathrm{geom}}_{\mathrm{Hecke}}( G)$ of $ \mathcal Z^{\mathrm{geom}}( G)$ consisting of transformations of the identity functor that commute with all Hecke operators $T_V$, for $V\in \mathcal Rep_{\mathrm{fin\text{-}proj}}(({}^LG)^I)$.

\subsubsection{Construction of semisimple parameters}
\label{subsubsection The construction of semisimple parameters}

We now recall, following \cite[Sec VIII.4]{FS-main}, the idea of V. Lafforgue for constructing L-parameters. In a sense, this explicates the morphism (\ref{Excursion-W to Z-geom-P}).

The main idea is to build, from functions in $\mathcal O(Z^1(W/P, \hat G))^{\hat G}$, an excursion datum, which in turn gives rise to an endomorphic transformation of the identity functor $  \mathrm {id}_{\mathcal C_P}$. To begin with, an {\bf excursion datum} is a quintuple 
$$\mathfrak d = (I, (w_i)_{i\in I}, V, \alpha, \beta),$$ where $I$ is a finite set, each $w_i$ belongs to $W$, $V$ is a finite projective representation of $({}^LG)^I$, and $\alpha: \mathbf 1_{\hat G} \to V|_{\hat G}$ and 
$\beta: V|_{\hat G}\to  \mathbf 1_{\hat G}$ are $\hat G$-linear morphisms. To a given function $f\in \mathcal O(Z^1(F_n, \hat G))^{\hat G}
$, we put $I=\{1,\dots,n\}$ and view $ \mathcal O(Z^1(F_n, \hat G))^{\hat G}$ as $ \mathcal O(({}^LG)^I/\Delta{\hat G})$, where $\Delta{\hat G}$ denoted the diagonal action of ${\hat G}$ on $({}^LG)^I$ by conjugations. Let $V\in \mathcal Rep(({}^LG)^I)$ be the subspace in $\mathcal O(({}^LG)^I/\Delta{\hat G})$ generated by $f$, and define $\alpha: 1\mapsto f$ and $\beta: f\mapsto f(1_{({}^LG)^I})$. For each $(w_i)\in W^I$, form an excursion datum $\mathfrak d = (I, (w_i)_{i\in I}, V, \alpha, \beta)$ as constructed from $f$. The functor of Hecke operators (\ref{the general functor connecting Galois rep and Bernstein center}) then defines, for each object $\pi $ in ${\mathcal C_P}$, 
an operator 
\begin{equation}
\label{explicit operator from excursion data}
{\mathfrak d}(\pi) :\pi\xrightarrow{T_\alpha(\pi)} T_V(\pi)\xrightarrow{(w_i)_{i\in I}} T_V(\pi)\xrightarrow{T_\beta(\pi)} \pi.
\end{equation}
It is routine to check various compatibility conditions, which result in the morphism (\ref{Excursion-W to Z-geom-P}), $f\mapsto \mathfrak d$.

We finally arrive at \cite[Cor VIII.4.3 or Def/Prop IX.4.1]{FS-main}: given 
$$\pi\in \mathcal Rep(G_b(F))^{\mathrm{sim}}_{/\mathrm{iso}} \hookrightarrow \mathcal Rep(G_b(F))\rightarrow \mathcal {D}_{\mathrm{lis}}(G_b(F)) =\mathcal {D}_{\mathrm{lis}}(\mathrm{Bun}^b_G) \ni \mathcal F_\pi ,$$ 
the action of $\mathcal Z^{\mathrm{geom}}(G)$ on $\mathcal F_\pi$ yields a character of $\pi_0 (\mathrm{End}(\mathrm{id}_{\mathcal C_P} ))$ by the (Schur-)irreducibility of $\pi$. The induced character of $\mathrm{Exc}(W_F,\hat G) $ by (\ref{Excursion-W to Z-geom-P}) then determines a point of $\mathrm{par}_G$, i.e., a semisimple character $\varphi_\pi^{\mathrm{ss}}\in (Z^1(F_n, \hat G)\sslash \hat G )(\Lambda)$ associated to $\pi$. More explicitly, given any excursion datum $\mathfrak d$, define for  $\varphi\in |\mathrm{Par}_G|(\Lambda)$ the scalar $\varphi(\mathfrak d)\in \Lambda$ by 
\begin{align*} 
\varphi(\mathfrak d)
&= (\Lambda \xrightarrow{\alpha} V \xrightarrow{(\varphi(w_i))_{i\in I}} V \xrightarrow{\beta} \Lambda ),
\end{align*}
then we have $\pi\mapsto \varphi$ if and only if $\varphi(\mathfrak d) = {\mathfrak d}(\pi)  $ as defined in (\ref{explicit operator from excursion data}), which is a scalar by the irreducibility of $\pi$, for all $\mathfrak d$.

\subsection{Fargues' conjecture}
\label{subsection Fargues conjecture}

We begin with some definitions. We call a parameter $\varphi:W_F\rightarrow \hat G$ {\bf elliptic} if $S_\varphi/(Z\hat G)^{W_F}$ is finite, and {\bf cuspidal} furthermore its restriction $\varphi|_{I_F}$ to the inertial subgroup of $W_F$ has a finite image, i.e., there is no non-trivial monodromy.

We can now state Fargues' conjectures. 
\begin{conj}
\label{Fargues-conjecture-detailed}
\cite[Conj 4.4]{Fargues-overview}. Given an elliptic parameter $\varphi$, there exists a complex of sheaves $\mathcal F_\varphi\in \mathcal D_{\mathrm{lis}}(\mathrm{Bun}_G)$, equipped with an action of $S_\varphi$, satisfying the following properties.

\begin{enumerate}[(i)]
\item (Hecke eigensheaf) Let $r_\mu$ be the ${}^LG$-representation appearing in Proposition \ref{MVG theorem of geometric Satake}, then 
$$T_{V_\mu} * \mathcal F_\varphi = \mathcal F_\varphi \boxtimes (r_\mu\circ \varphi) \in \mathcal D_{\mathrm{lis}}(\mathrm{Bun}_G )^{ B(W_F)^I}.$$  
In other words, 
$\mathcal F_\varphi$ is an eigensheaf of $T_{V_\mu} $ with eigenvalue $r_\mu\circ\varphi$.

\item (Cuspidality) Put $i^{\mathrm{ss}}:\mathrm{Bun}^{\mathrm{ss}}_G := \bigsqcup_{b\in B(G)_{\mathrm{bas}}} \mathrm{Bun}^b_G \hookrightarrow \mathrm{Bun}_G $.  If $\varphi$ is cuspidal, then $i^{\mathrm{ss}}_{\natural}i^{\mathrm{ss},*} \mathcal F_\varphi = \mathcal F_\varphi $.

\item (L-packet) For $b\in B(G)_{\mathrm{bas}}$, we have
$$i^{b,*}\mathcal F_\varphi \cong \bigoplus_{} (\dim \rho) \pi_{\varphi,\rho} \in \mathcal Rep( G_b(F)),$$
where $\rho$ ranges over irreducible representations of $S_\varphi$ such that $\rho|_{Z\hat G^{\Gamma_F}}$ corresponds to $b\in \pi_1(G)_{\Gamma_F}$ (via $\kappa$ and (\ref{Kottwitz duality})), and $\pi_{\varphi,\rho}$ is the representation via the conjectural LLC for $G_b$.

 \item (Conjectural construction) Put $x_\varphi: *  \rightarrow \mathrm{Par}_G$ as the point corresponding to $\varphi$, and denote 
$\mathcal E_\varphi  = x_{\varphi, *}\underline{\Lambda} \in \mathcal{QC}oh(\mathrm{Par}_G) = \mathrm{Ind}\mathcal{P}erf^{\mathrm{qc}}(\mathrm{Par}_G)$, then 
$$\mathcal F_\varphi = \mathcal E_\varphi * \mathcal W_{\mathfrak f}.$$
Here $\underline{\Lambda}$ is the skyscraper sheaf supported at $x_\varphi$.  \qed
\end{enumerate}
\end{conj}
The last two parts propose a very natural construction of $\mathcal F_\varphi$ and its associated L-packets, which is, in a sense, the converse of the main statement of the CLLC. Since the image of $x_\varphi$ is $[*/S_\varphi] \hookrightarrow \mathrm{Par}_G$, the pushforward $x_{\varphi,*}\bar{\mathbb Q}_\ell $ is indeed the induction of representations from the trivial group to $S_\varphi$, i.e., 
\begin{equation}
\label{induced regular representation of S_varphi}
\mathcal E_{\varphi } = \mathrm{Ind}_1^{S_\varphi}\Lambda \cong \bigoplus_{\rho \in \mathcal Rep(S_\varphi)^{\mathrm{sim}}_{/\mathrm{iso}}} \rho^\vee \otimes \rho.
\end{equation}
Regarding each $\rho$ as a skyscraper sheaf supported at the point $x_\varphi \in \mathrm{Par}_G$, we  further conjecture that 
\begin{equation}
\label{Fargues-conjecture-individual}
i^{b,*}(\rho * \mathcal W_{\mathfrak f}) = \pi_{\varphi,\rho}.
\end{equation}
See also \cite[Conj 2.1.8]{Hansen-Beijing-notes} and \cite[Conj X.2.2]{FS-main} for some alternative formulations.

One may wonder whether an analogue of (\ref{Fargues-conjecture-individual}) also holds for non-elliptic parameters (if we use $x_{\varphi ,!}\Lambda$ instead of $\mathcal E_\varphi  = x_{\varphi, *}\Lambda$). In one manageable case, if $\varphi$ satisfies a certain regularity condition known as being \emph{generous} \cite[Def 2.1.5 and Conj 2.1.9]{Hansen-Beijing-notes} (similar to the notion of a generic representation of $G(F)$, as the terminology suggests), then we can construct a Hecke eigensheaf $\mathcal F_\varphi$ similar to Conjecture \ref{Fargues-conjecture-detailed}(iv). In the general non-elliptic case when $x_\varphi$ is not a closed point, the (micro-)local geometry around it may be very complicated; for instance, it may exhibit non-trivial fibres of the cotangent complex of $\mathrm{Par}_G$.

\subsection{Parabolic induction}

We explain that Hecke operators are compatible with parabolic induction. Indeed, we explain this compatibility in terms of the excursion map, as follows.

Let $M$ be a Levi subgroup of $G$ contained in a parabolic $P$. We view $P$ as a dynamical parabolic and choose $\mu\in X_*(T)$ associated with $P$, i.e., 
$$P = \{g\in G: \lim_{t\to 0}\mathrm{Ad}(\mu(t)) g\text{ exists}\}.$$
Put $b = \mu(\varpi)\in B(G)$, so that $G_b = M$. We define a transfer of the spectral centers $\mathcal Z^{\mathrm{spec}}( G)\to \mathcal Z^{\mathrm{spec}}( M)$ using the map $Z^1(M,W)\to Z^1(G,W)$ defined by $\varphi \mapsto \varphi_M$, where 
\begin{equation}
\label{descent of Langlands parameter to Gb}
\varphi_M:w\mapsto (2 \rho_{\hat G} - 2 \rho_{\hat G_b})(\sqrt{q})^{|w|}\varphi(w).
\end{equation}
Let $\pi$ be a representation of $M(F)$, which may be assumed to be compactly generated as a complex of representations. Consider 
\begin{equation}
\label{representation corresponds to sheaves, for different isocrystals}
\pi \in \mathcal Rep (G^b(F)) \xrightarrow{\sim} \mathcal D(\mathrm{Bun}_G^b) \ni \mathcal F = \mathcal F_\pi ,
\end{equation}
i.e., $\mathcal F$ corresponds to $\pi$ under the above equivalence of categories \cite[Prop VII.7.1]{FS-main}, push it forward to $D(\mathrm{Bun}_G)$, and continue to denote it by $\mathcal F$. Then the composition
\begin{equation}
\label{composition of excursion via Levi subgroup}
\mathcal Z^{\mathrm{spec}}( G)\to \mathcal Z^{\mathrm{spec}}( M) \xrightarrow{\text{(\ref{Excursion-W to Z-geom-P}) to $\pi$}} \mathrm{End}(\pi) \to \mathrm{End}(\mathrm{Ind}_{P(F)}^{G(F)}\pi),
\end{equation}
is equal to the excursion map (\ref{Excursion-W to Z-geom-P}) applied to $\mathrm{Ind}_{P(F)}^{G(F)}\pi$ \cite[Cor IX.7.3]{FS-main}. Here we use the unnormalized parabolic induction of $\pi$, which corresponds to $T_V(\mathcal F)|_{\mathrm{Bun}^1_G}$ under (\ref{representation corresponds to sheaves, for different isocrystals}) with $b=1$, and $V$ taken to be the highest weight representation of weight $\mu$.

Finally, we remark that the transfer of parameters (\ref{descent of Langlands parameter to Gb}) can be defined for general $b\in B(G)$. This is explained in the proof of \cite[Th IX.7.2]{FS-main}, which computes explicitly the compatibility formula of the excursion map (\ref{composition of excursion via Levi subgroup}) with the transfer from $G$ to $G_b$. We omit the technical details.

\subsubsection{Example: $\mathrm{GL}_2$}
\label{subsubsection Example: GL2}

We have the following exact sequence of representations of $G(F)$ with $G = \mathrm{GL}_2$ \cite[(9.10.4)]{BH-GL2} \begin{align}
0\to  \mathrm{St}_G \to  \mathrm{Ind}_B^G \delta_B^{-1} \to \mathbf 1_G \to 0 .
\label{trivial rep sequence in PGL2}
\end{align}
The parameters corresponding to $\mathbf 1_G$ and $\mathrm{St}_G$ are respectively
\begin{align}
&\varphi_0: \dot{\mathrm{Fr}} \mapsto \mathrm{diag}(q^{1/2},q^{-1/2}), \quad t\mapsto \mathrm{diag}(0,0),\text{ and}
\\
&\varphi_1: \dot{\mathrm{Fr}} \mapsto \mathrm{diag}(q^{1/2}, q^{-1/2}), \quad t\mapsto   \begin{bmatrix} 
      1 & 1 \\
       & 1 \\
   \end{bmatrix}.
\end{align}
Let's examine the geometry around $ x_{\varphi_0}$ and $x_{\varphi_1}$. The points $ x_{\varphi_0}$ and $x_{\varphi_1}$ have the same image in $|\mathrm{Par}_G|$. Here we have non-trivial monodromy, which renders the cotangent bundle at $x = | x_{\varphi_0}|$ non-trivial, i.e., we are out of the generous case as discussed at the end of Section \ref{subsection Fargues conjecture}. Moreover, the point $x_{\varphi_0}$ is closed, whereas $x_{\varphi_1}$ is open and $\{x_{\varphi_1}\}^-=C_x:=\{ x_{\varphi_0} ,x_{\varphi_1}\}\subset \mathrm{Par}_G$. For $i\in \{0,1\}$, denote by $j_i: x_{\varphi_i}\to C_x$ the inclusion, and let $\mathbf 1_i$ be the trivial local system supported at the point $x_{\varphi_i}$. Consider the following sequence in $\mathcal Perf(\mathrm{Par}_G)$,
\begin{equation}
\label{Steinberg sequence in D(ParG)}
j_{1*}\mathbf 1_1 \to j_{0!}\mathbf 1_0 \to j_{1!}\mathbf 1_1 [1],
\end{equation}
which is rotated from the more natural one:
$$ j_{1!}\mathbf 1_1 \to j_{1*}\mathbf 1_1  = \underline{\Lambda}_{C_x}\to j_{0!}\mathbf 1_0 \to. $$
The calculation from \cite[Th 4.43]{Hellmann-IwahoriHecke} suggests that applying the spectral action of (\ref{Steinberg sequence in D(ParG)}) to $\mathcal W_{\mathfrak f}$ should yield (\ref{trivial rep sequence in PGL2}).

\section{Arthur packets and sheaves}

Recall that $Z^1(W_F , \hat G(\Lambda)):=Z^1(W_F , \hat G)(\Lambda)$ is the set of continuous 1-cocycles $\lambda: W_F\to \hat G(\Lambda)$. We call $\lambda$ an {\bf infinitesimal parameter} if it is smooth, i.e., locally constant, and denote by $\Lambda(G/F)$ the set of equivalence classes of infinitesimal parameters of $G$ under $\hat G(\Lambda)$-conjugacy.

\subsection{Arthur parameters}

When $\Lambda=\mathbb C$, by the so-called `no small subgroups' property, any $\lambda\in Z^1(W_F , \hat G(\Lambda))$ is automatically smooth. We put $\mathrm{SL}_2^D:=\mathrm{SL}_2(
\Lambda)$ and define a {Langlands parameter} as a morphism $W_F\times \mathrm{SL}_2^D\to {}^L G$ such that $\varphi|_{W_F} $ is smooth and the restriction $\varphi :{\mathrm{SL}^D_2 } \to \hat G(\Lambda)$ is algebraic over $\Lambda$. Denote by $\Phi(G/F)$ the set of equivalence classes of Langlands parameters of $G$ under $\hat G(\Lambda)$-conjugacy.

Let $\mathrm{SL}_2^A:=\mathrm{SL}_2^D$ denote another copy of $\mathrm{SL}_2(
\Lambda)$ but playing a different role. 
\begin{dfn}
An {\bf Arthur parameter} for $G$ is a group morphism $\psi : W_F\times \mathrm{SL}^D_2 \times \mathrm{SL}^A_2 \rightarrow {}^LG$ such that $\psi|_{W_F\times \mathrm{SL}^D_2} $ is a Langlands parameter, $\psi |_{\mathrm{SL}^A_2 }$ is algebraic over $\Lambda$, and $\psi({W_F})$ is bounded in $\hat  G(\Lambda)$ (i.e., its closure is compact). Denote by $\Psi(G/F)$ the set of equivalence classes of Langlands parameters of $G$ under $\hat G(\Lambda)$-conjugacy.
\qed\end{dfn}
We therefore have a composition of `forgetful' maps: 
\begin{equation}
\label{forgetful maps on parameters}
   \begin{split}
     &\Psi(G/F)\rightarrow \Phi(G/F) \rightarrow \Lambda(G/F), \quad 
     \\
 &\lambda_\psi(w):=\lambda_{\phi_\psi}(w) = \phi_\psi(w, d_w)  = \psi(w, d_w, d_{(w, d_w)}),
   \end{split}
\end{equation}
where $d_w = \mathrm{diag}(|w|^{1/2}, |w|^{-1/2}) \in \mathrm{SL}_2^D$ and $d_{(w,a)} = d_w \in \mathrm{SL}_2^A$. The map $\Phi(G/F) \rightarrow \Lambda(G/F)$ is clearly surjective, and \cite[Lem 3.3]{Cunningham-Voganish-begins} shows that $\Psi(G/F)\rightarrow \Phi(G/F) $ is injective. We say that $\varphi$ is {\bf of Arthur type} $\psi$ if $\psi\mapsto \varphi$ under the above map.

When $G$ is an orthogonal or a symplectic group, an important theorem \cite[Th 2.2.1]{Arthur-book} asserts that there is a finite multiset $\tilde\Pi_\psi$ over the set $\tilde \Pi_{\mathrm{unit}}(G(F))$ of irreducible unitary representations of $G(F)$, equipped with a
canonical mapping
$$\Psi_\psi: \tilde\Pi_\psi\rightarrow \hat{\mathcal S}_\psi, \quad \pi \mapsto \left<\cdot,\pi\right>,$$
and satisfying the following properties: 

\begin{enumerate}[(i)]
\item $\tilde\Pi_\psi$ is constructed from $\psi$ using endoscopic transfer (i.e., representations in $\tilde\Pi_\psi$ and the map $\Psi_\psi$ satisfy an endoscopic character identity);

\item $\tilde\Pi_\psi$  contains the L-packet $\tilde\Pi_{\phi_\psi}$; 

\item $\tilde\Pi_\psi$ is a set (multiplicity one) if $G$ is quasi-split. 

\end{enumerate}
The multiset $\tilde\Pi_\psi$ is commonly known as an Arthur packet. In \cite{Cunningham-Voganish-begins}, they show by examples that certain perverse sheaves supported on the fibre in $\mathrm{Par}_G$ of $\lambda$ satisfy linear relations analogous to the endoscopic character identity relations on representations in $\tilde\Pi_\psi$, thereby establishing a sheaf-representation correspondence. We will review how these perverse sheaves are constructed in Section \ref{subsection Conjectural Arthur packet in CLLC}.

When $\Lambda=\overline{\mathbb Q}_\ell$, the discussion is somewhat different. A continuous 1-cocycle $\lambda: W_F\to \hat G(\Lambda)$ is not necessarily smooth, but we have the following result from Grothendieck's quasi-unipotence theorem \cite[Prop VIII.2.5]{FS-main}. Knowing that $I_F/P_F \cong \prod_{m\neq p}\mathbb Z_{m} $, we choose a lifting $t: \mathbb Z_\ell \rightarrow I_F$, and embed $\hat G$ into some $\mathrm{GL}_m$ (all choices will eventually become irrelevant), then there exists a nilpotent $N\in \hat {\mathfrak g}(\Lambda) $ satisfying 
\begin{align} 
& \mathrm{Ad}(\lambda(w)) N = |w| N, \quad w\in W_F,
\label{Action of Frobenius on the Deligne part of N}
\\
\text{ and }\quad &\lambda(t(x)) = \exp(x N), \quad x\in \ell^k \mathbb Z_\ell\text{ for }k\gg 0.
\label{Grothendieck quasi-unipotence element}
\end{align}
The parameter defined by 
$$\dot{\mathrm{ Fr} }^m \tau \mapsto \lambda(\dot{\mathrm{ Fr} }^m \tau ) \exp(-\tau_\ell N), \quad \tau\in I_F,\,m\in \mathbb Z,$$    
for a choice of Frobenius $\dot{\mathrm{ Fr} }$ (again, eventually irrelevant) and with $\tau_\ell $ being the projection of $\tau$ via $I_F/P_F \twoheadrightarrow \mathbb Z_{\ell} $, is then smooth.

To unify the discussion, for any algebraically closed field $\Lambda$ of characteristic 0, we define a {\bf Langlands parameter} over $\Lambda $ to be a pair $\varphi = (\lambda, N)$, where $\lambda$ is an infinitesimal parameter and $N\in \hat {\mathfrak g}(\Lambda) $ satisfying (\ref{Action of Frobenius on the Deligne part of N}) and 
(\ref{Grothendieck quasi-unipotence element}). Denote by $\Phi(G/F)$ the set of equivalence classes of Langlands parameters under $\hat G(\Lambda)$-conjugacy.
\begin{prop}
\begin{enumerate}[(i)]
\item $\Phi(G/F)$ coincides with the previous constructions when $\Lambda = \mathbb C$ or $\overline{\mathbb Q}_\ell$.

\item (When $\Lambda$ is of characteristic 0,) $N$ is necessarily nilpotent.

\end{enumerate}

\end{prop}
\proof 
Both facts can be found in \cite{Imai-Langlands-parameters}. (i) follows directly from \emph{loc. cit.}, Prop 1.7 and 1.13 for $\Lambda = \mathbb C$ and \emph{loc. cit.}, Prop 1.17 for $\Lambda = \overline{\mathbb Q}_\ell$, while (ii) is  \emph{loc. cit.}, Lem 1.6.
\qed

When $\Lambda = \overline{\mathbb Q}_\ell$, we will define in Section \ref{subsection The singularity stack}
the singularity stack of $\mathrm{Par}_G$ and show that it is a good substitute for $\Psi(G/F)$ in the case $\Lambda=\mathbb C$.

\subsection{Regular conormal bundles}

We will see that Arthur packets are constructed by pushing-forward sheaves to $\mathrm{Par}_G$ from its conormal bundles. We first recall from \cite{Cunningham-Voganish-begins} how this is done when $\Lambda = \mathbb C$, and then propose an analogous construction for $\Lambda = \overline{\mathbb Q}_\ell$.

First put $\Lambda = \mathbb C$. Given an infinitesimal parameter $\lambda$, define the {\bf Vogan variety} 
$$V_\lambda\text{ (resp. ${}^tV_\lambda$)}:=\{ x \in  \mathrm{Lie} Z_{\hat G}(\lambda(I_F)) : \mathrm{Ad}(\lambda (\dot{\mathrm{ Fr} }) ) x = q x\text{ (resp. $q^{-1} x$})\},$$
and denote by $H_\lambda$ the stabilizer of $\lambda$ in $\hat G(\Lambda)$. Clearly, $H_\lambda$ acts on $V_\lambda$ by conjugation. Recall the `forgetful' map $  \Phi(G/F) \rightarrow \Lambda(G/F)$ from (\ref{forgetful maps on parameters}), and denote by $\Phi_\lambda(G/F)$ the preimage of $\lambda$ under this map. By \cite[Sec 4.3]{Cunningham-Voganish-begins}, $\Phi_\lambda(G/F)$ canonically parametrizes the orbit space $V_\lambda/ H_\lambda$ bijectively. Moreover, there is a partial order on $V_\lambda/ H_\lambda$: 
$$C\leq C' 
\quad \Leftrightarrow \quad 
C\subset \overline{C'},$$
which induces a partial order on $\Phi_\lambda(G/F)$ via the canonical bijection $\varphi\leftrightarrow C_\varphi$ above.

Let $C$ be an $H_\lambda$-orbit of $V_\lambda$. Define, as a subset of the cotangent bundle $T^*(V_\lambda)  = V_\lambda \times {}^tV_\lambda$, the conormal bundle along $C$ as 
\begin{equation}
\label{conormal bundle description}
T_C^*(V_\lambda) = \{(x,\xi) \in  T^*(V_\lambda): x\in  C\text{ and } [x,\xi]= 0\},
\end{equation}
and the {\bf regular conormal bundle} associated with $C$, i.e., the regular part of $T_C^*(V_\lambda)$, as
\begin{equation}
\label{regular conormal bundle, definition}
\mathbb O_C:= T_C^*(V_\lambda)_{\mathrm{reg}} = T_C^*(V_\lambda) \setminus \bigcup_{C\subsetneq \overline{C'}}T_{C'}^*(V_\lambda),
\end{equation}
and put $\mathbb O_\varphi = \mathbb O_{C_\varphi}$.  Note that the $H_\lambda$-actions on $V_\lambda$ and $V_\lambda^*$ induce corresponding actions on $\mathbb O_\varphi$.

Dually, we can also define, for each $H_\lambda$-orbit $B$ of $V_\lambda^*$, the conormal bundle $T_B^*(V_\lambda^*)$ along $B$. Pyasetskii duality \cite[Cor 2]{Pyasetskii:1975aa} asserts that, for each $H_\lambda$-orbit $C$ of $V_\lambda$, there is a unique $H_\lambda$-orbit $C^*$ of $V_\lambda^*$ such that 
$$ \overline{T_C^*(V_\lambda)} =  \overline{T_{C^*}^*(V^*_\lambda)},  $$
thereby inducing a duality bijection between  $V_\lambda/H_\lambda$ and $V_\lambda^*/H_\lambda$.

Denote by $A_{C_\varphi}$ (resp. $A_{\mathbb O_\varphi}$) the equivariant fundamental group of $C_\varphi$ (resp. $ \mathbb O_\varphi$) (i.e., the $\pi_1$-group defined using $H_\lambda$-equivariant coverings). The projections $C_\varphi\twoheadleftarrow \mathbb O_\varphi\twoheadrightarrow  C^*_\varphi$ induce group morphisms
$A_{C_\varphi}\leftarrow A_{\mathbb O_\varphi}\rightarrow  A_{C^*_\varphi}$, which are again surjective.

Given $\varphi\in \Phi(G/F)$ (resp. $\psi\in \Psi(G/F)$), we put $A_\varphi: = \pi_0(Z_{\hat G}(\varphi))$ (resp. $A_\psi: = \pi_0(Z_{\hat G}(\psi))$). If 
$\psi\mapsto \varphi\mapsto \lambda$ under (\ref{forgetful maps on parameters}), we can interpret $A_\varphi$ as $\pi_0(Z_{H_\lambda}C_\varphi)$ and $A_\psi$ as $\pi_0(Z_{H_\lambda}\mathbb O_\varphi)$, where $Z_{H_\lambda}C_\varphi$ is the ${H_\lambda}$ stabilizer of any point in $C_\varphi$, and $Z_{H_\lambda}\mathbb O_\varphi$ likewise. \cite[Lem 4.4 (resp. Prop 6.11)]{Cunningham-Voganish-begins} then  implies that $A_{C_\varphi} \cong A_\varphi$ (resp. $A_{\mathbb O_\varphi}\cong A_\psi$ if $\varphi$ is of Arthur type $\psi$). Therefore, it induces an equivalence of categories
$$\mathcal{LS}([C_\varphi/H_\lambda]) \xrightarrow{\sim} \mathcal Rep(A_\varphi)\quad\text{ (resp. }\mathcal{LS}([\mathbb O_\varphi/H_\lambda]) \xrightarrow{\sim} \mathcal Rep(A_\psi)).$$

\begin{conj}
$\mathcal Perf( C_\varphi)  \simeq   \mathcal D^b([*/ A_\varphi]) \simeq   \mathcal Rep_{\mathrm{fin}}(A_\varphi)$. Moreover, $\mathcal Perf( \mathbb O_\varphi) \simeq   \mathcal D^b([*/ A_\psi ]) \simeq  \mathcal Rep_{\mathrm{fin}}(A_\psi)$ whenever $\varphi$ is of Arthur type $\psi$.
\qed\end{conj}

\subsection{The singularity stack}
\label{subsection The singularity stack}

We now consider a similar construction when $\Lambda = \overline{\mathbb Q}_\ell$. We again adopt the convention $\varphi=(\lambda,N)\in \mathrm{Par}_G$ for Langlands parameters. Let $\mathbb L_{\mathrm{Par}_G/\Lambda}$ be the cotangent complex of $\mathrm{Par}_G$. Since $\mathrm{Par}_G/\Lambda$ is a local
complete intersection, $\mathbb L_{\mathrm{Par}_G/\Lambda}$ is a perfect complex with $H^i(\mathbb L_{\mathrm{Par}_G/\Lambda}) =0$ for all $i$ except $i=-1$ or 0 \cite[Prop 92.14.4]{stacks-project}, i.e., $\mathbb L_{\mathrm{Par}_G/\Lambda}$ is supported only in degrees $i\in \{-1,0,1\}$.

\begin{prop}\label{pulling back Cotangent space}
\cite[Th 5.4]{FS-notes}, \cite[Cor VIII.2.3]{FS-main} Let $x_\varphi$ be the $\Lambda$-point corresponding to $\varphi \in \mathrm{Par}_G$, then the fibre
$$x_\varphi^*(\mathbb L_{\mathrm{Par}_G/\Lambda}) = 
 \Gamma(W_F,\hat{\mathfrak g}^*_\varphi(1))[1].$$
Here $\hat{\mathfrak g}^*_\varphi(1)$ denotes the $W_F$-fixed point space of $\hat{\mathfrak g}^*$ equipped with the $W_F$-action 
$$(\lambda(\dot{\mathrm{ Fr} }), v)\mapsto   q \mathrm{Ad} (\lambda(\dot{\mathrm{ Fr} })) ({}^{\dot{\mathrm{ Fr} } } v ), \quad v\in \hat{\mathfrak g}^*, $$
and which is trivial on $I_F$.
\qed\end{prop}

The {\bf singularity stack}
$$\mathrm{Sing}_{\mathrm{Par}_G/\Lambda}:= \mathrm{Spec}(\mathrm{Sym}^\bullet_{\mathcal O_{\mathrm{Par}_G}}
H^1(\mathbb L^\vee_{
\mathrm{Par}_G/\Lambda} )) \rightarrow \mathrm{Par}_G$$ 
is the scheme over $\mathrm{Par}_G$ that represents the functor 
$$T /\mathrm{Par}_G \mapsto {H}^{-1} (\mathbb L_{\mathrm{Par}_G/\Lambda}\otimes_{\mathcal O_{\mathrm{Par}_G} }\mathcal O_T).$$
Concretely, if $\mathbb L_{\mathrm{Par}_G/\Lambda}$ is represented by a complex $\mathcal E^{-1}\rightarrow \mathcal E^{0}\rightarrow \mathcal E^{1}$ of sheaves in the derived category, then  
$\mathrm{Sing}_{\mathrm{Par}_G/\Lambda} $ represents the module $\ker( \mathcal E^{-1}_x\rightarrow \mathcal E^{0}_x)$ at $x\in \mathrm{Par}_G$. Hence $\mathrm{Sing}_{\mathrm{Par}_G/\Lambda}$ measures the local singularity structure of $\mathrm{Par}_G$.

\begin{rmk}$\mathrm{Sing}_{\mathrm{Par}_G/\Lambda}$ is denoted by $\mathrm{Arth}_{\hat G} $ in \cite{Arinkin-Gaitsgory-2015}. It plays the role of Arthur parameters in their setting.
\qed\end{rmk}

Proposition \ref{pulling back Cotangent space}
 hence implies a natural embedding
$$\mathrm{Sing}_{\mathrm{Par}_G/\Lambda} \rightarrow [\hat{\mathfrak g}^*/\hat G] \times_{[*/\hat G]}{\mathrm{Par}_G/\Lambda}$$
defined over $\mathrm{Par}_G$. Indeed, one obtains a better described image in $\hat{\mathfrak g}^*$.

\begin{prop}
\label{nilpotent singular support condition}
\cite[Prop VIII.2.11]{FS-notes} Nilpotent singular support condition. With $\Lambda = \overline{\mathbb Q}_\ell$, for any $x\in \mathrm{Par}_G$, we have 
$$x^*( \mathrm{Sing}_{\mathrm{Par}_G/\Lambda})\subset \mathcal N_{\hat{\mathfrak g}^*},$$ 
where $\mathcal N_{\hat{\mathfrak g}^*}$
is the nilpotent cone in ${\hat{\mathfrak g}^*}$.
\qed\end{prop}

Fix an infinitesimal parameter $\lambda$, and denote by $\mathrm{Par}_\lambda\hookrightarrow \mathrm{Par}_G$ be the fibre over $\lambda$. For each $N \in \hat{\mathfrak{g}}$ satisfying (\ref{Action of Frobenius on the Deligne part of N}) and (\ref{Grothendieck quasi-unipotence element}), let $C=C_N\hookrightarrow \mathrm{Par}_\lambda$ be the $H_\lambda$-orbit of $N\in \hat{\mathfrak g}$.
Using an $\mathrm{Ad}(\hat G)$-equivariant identification $\hat{\mathfrak g}\cong \hat{\mathfrak g}^*$, the fibre
$$\mathrm{Sing}_{\lambda,C}:= \mathrm{Sing}_{\mathrm{Par}_G}|_{\mathrm{Par}_\lambda}\times_{\mathrm{Par}_\lambda}C,$$
then admits a concrete description: 
\begin{equation}
\label{sing lambda C}
\mathrm{Sing}_{\lambda,C} =\{(N,X): N\in C\text{ and } X\in \hat{\mathfrak g}^*\text{ such that }[N,X]=0\},
\end{equation}
which is parallel with (\ref{conormal bundle description}). Combining this with Proposition \ref{nilpotent singular support condition}, we see that $\mathrm{Sing}_{\lambda,C}$ plays an analogous role to the conormal bundle $T^*_C(V_\lambda)$ in the case $\Lambda = \mathbb C$.

We remark that, using (\ref{sing lambda C}), we can analogously define the regular part 
$(\mathrm{Sing}_{\lambda,C})_{\mathrm{reg}}$ of $\mathrm{Sing}_{\lambda,C}$ as in (\ref{regular conormal bundle, definition}).

\subsection{Conjectural Arthur packets}
\label{subsection Conjectural Arthur packet in CLLC}

When $\Lambda=\mathbb C$, Arthur packets can be constructed using the vanishing cycle functors from sheaves on Vogan varieties to those on regular conormal bundles. In this section, we recall the construction from \cite{Cunningham-Voganish-begins}, and discuss how we can generalize this construction to the case where $\Lambda=\overline{\mathbb Q}_\ell$.

Let $C$ be an $H_\lambda$-orbit in $V_\lambda$, and $\mathcal D_c^b([V_\lambda / H_\lambda])_C$ be the full subcategory of 
$\mathcal D_c^b([V_\lambda / H_\lambda])$ consisting of complexes of constructible sheaves whose supports contain $C$ \cite[Prop 7.10]{Cunningham-Voganish-begins}. Recall the functor defined in \cite[(7.11)]{Cunningham-Voganish-begins},
\begin{equation}
\label{vanishing cycle functor, C-case}
\mathrm{Ev}_C:\mathcal D^b([V_\lambda / H_\lambda])_C \rightarrow \mathcal D^b([\mathbb O_C/ H_\lambda]\times_\Lambda \mathbb S),
\end{equation}
where $\mathbb S$ denotes the trait over $\Lambda$. We will recall its construction in Subsection \ref{subsubsection Vanishing cycle functors} as an appendix. As an important property,  \cite[Prop 7.20]{Cunningham-Voganish-begins} shows that $\mathrm{Ev}_C$ maps perverse sheaves on $[V_\lambda / H_\lambda]$ to perverse sheaves on $[\mathbb O_C / H_\lambda]$.

{\color{red}

Denote the right adjoint of $\mathrm{Ev}_C$ by 
$$\mathrm{Ev}_C^*:\mathcal D^b([\mathbb O_C/ H_\lambda]\times_\Lambda \mathbb S) \to \mathcal D^b([V_\lambda / H_\lambda])_C \hookrightarrow \mathcal D^b([V_\lambda / H_\lambda]).$$
 Let $\underline{\Lambda}_{}$ be the skyscraper sheaf supported on $C$ with values in $\Lambda$.
\begin{dfn}
\label{Arthur sheaf Lie algebra}
We call $\mathcal A_C = \mathrm{Ev}_C^* (\underline{\Lambda}_{[\mathbb O_C / H_\lambda]\times_\Lambda \mathbb S} )$ an {\bf Arthur sheaf}. 
\qed\end{dfn}
Compare this definition with the construction of Fargues sheaves in Conjecture \ref{Fargues-conjecture-detailed}.

If $\mathrm{Ev}_C \mathcal P$ is of finite rank for all $\mathcal P \in \mathcal Per([V_\lambda / H_\lambda])^{\mathrm{sim}}_{/\mathrm{iso}}$, the set of isomorphism classes of simple objects in $\mathcal Per([V_\lambda / H_\lambda])$, and if this set is finite, then by induction (compare with (\ref{induced regular representation of S_varphi})) we can express the Arthur sheaf as a finite direct sum: 
$$\mathcal A_C \cong  \bigoplus_{\mathcal P \in \mathcal Per([V_\lambda / H_\lambda])^{\mathrm{sim}}_{/\mathrm{iso}}
}\mathcal P^{\oplus \mathrm{rank}(\mathrm{Ev}_C \mathcal P)},$$  
which agrees with the description in \cite[(10.15)]{Cunningham-Voganish-begins}. In particular, we see that $\mathcal A_C$ is a perfect sheaf under the above assumptions, which are satisfied by the examples in \emph{loc. cit.}.

When $\Lambda = \overline{\mathbb Q}_\ell$, we mimic (\ref{vanishing cycle functor, C-case}) and obtain a functor, which we still denote by $\mathrm{Ev}_C$,  
$$\mathrm{Ev}_C:\mathcal D^b([V_\lambda / H_\lambda])_C \rightarrow \mathcal D^b([ (\mathrm{Sing}_{\lambda,C})_{\mathrm{reg}} / H_\lambda]\times_\Lambda \mathbb S),$$ 
and define the Arthur sheaf $\mathcal A_C$ just as in Definition \ref{Arthur sheaf Lie algebra}.
\begin{conj}
\label{main conjecture}
\begin{enumerate}[(i)]
\item $\mathcal A_C$ is a perfect complex (see \cite[Cor VIII.2.10]{FS-main}).

\item The Arthur packet $\Pi_\psi$ is given by $\mathcal A_C * \mathcal W_{\mathfrak f}$. \qedhere

\end{enumerate}
\end{conj}

The perfectness condition on $\mathcal A_C$ is required in order to fit into the Fargues-Scholze framework, as we saw in the context of the Hecke spectral action. With the assumptions after Definition \ref{Arthur sheaf Lie algebra}, this condition on $\mathcal A_C$ is automatic. In general, Conjecture \ref{main conjecture}(i) should follow from Arinkin-Gaitsgory's Theorem, which will be stated in Proposition \ref{Arinkin-Gaitsgory's Theorem} below.

}

Let $X/S$ be syntomic, i.e., flat and a local complete intersection. Let $\mathbb L_{X/S}$ be the cotangent complex and $\mathrm{Sing}_{X/S} := \mathrm{Spec}(\mathrm{Sym}^\bullet_{\mathcal O_X} H^{-1}(\mathbb L_{X/S}) ^\vee) \rightarrow X$ be the stack of singularities. For $\mathcal E\in \mathcal Coh^b(X)$, denote its singular support \cite[Def VIII.2.8]{FS-main} by
$\mathrm{SingSupp}(\mathcal E)\subset \mathrm{Sing}_{X/S}$.
\begin{prop}
\label{Arinkin-Gaitsgory's Theorem}
(\cite[Th. 4.2.6]{Arinkin-Gaitsgory-2015}, see also \cite[Th VIII.2.9]{FS-main}). $\mathcal E$ is a perfect complex if and only if  $\mathrm{SingSupp}(\mathcal E)$ is contained in the zero section (in the sheaf ${H}^{-1} (\mathbb L_{X/S})$ over $X$) of $\mathrm{Sing}_{X/S}$. 
\qed\end{prop}

\subsubsection{Vanishing cycle functors}
\label{subsubsection Vanishing cycle functors}

Let $C$ be an $H_\lambda$-orbit of $V_\lambda$. In this appendix, we recall the definition of the functor in (\ref{vanishing cycle functor, C-case}) from \cite[(7.11)]{Cunningham-Voganish-begins}. The original setup can be traced back to \cite[Expos{\'e} XIII]{SGA7II}.

We recall the trait $\mathbb S := \mathrm{Spec}(\Lambda[[t]])$ consisting of a generic point $\eta$ and a closed point $z$, as in the following diagram:
$$i:z = \mathrm{Spec}(\Lambda) \rightarrow \mathbb S \leftarrow \eta = \mathrm{Spec}(\Lambda((t)) ) : j.$$
Let $t: \mathbb S \to \mathbb A^1$ be the local coordinate, which corresponds to the canonical injection $\Lambda [t] \hookrightarrow \Lambda[[t]]$.

Let $X$ be an algebraic stack over $\Lambda$, and view $X $ as $X\times z$. Define the functor $R\Phi_X $ as the cone of the transformation defined by the following adjunction:
$$R\Phi_X:=\mathrm{cone}( i^* \to i^* j_* j^*).$$
For any sheaf $\mathcal F$  on $X$, the resulting sheaf $R\Phi_X\mathcal F$ can be defined on $X\times_z \mathbb S$. We remark that, since our base field $\Lambda$ is assumed to be algebraically closed, we do not need to pass through any Galois invariant subsheaves as done in \emph{loc. cit}.

Let $\mathcal F$ be a sheaf on the quotient stack $[V_\lambda/H_\lambda]$, viewed as an $H_\lambda$-equivariant sheaf on $V_\lambda$. Put $U = \mathrm{supp}\mathcal F \times C^*$, and let $f = \left<\cdot ,\cdot \right>|_{U}:U\to \mathbb A^1$ denote the restriction to $U$ of the canonical dual pairing of $\hat{ \mathfrak {g}}\times \hat {\mathfrak {g}}^*$. Put $X = U\times_{(f,\mathbb A^1,t)}\mathbb S$, and let $\pi: X\to U$ denote the projection map. For any sheaf $\mathcal G$ on $U$, we define 
$$R\Phi_f( \mathcal G): = R\Phi_{X}(\pi^* \mathcal G).$$
Now put $\mathcal G = \mathcal F|_{\mathrm{supp}\mathcal F} \boxtimes \mathbf 1_{C^*}$. The vanishing cycle functor is defined by 
$$\mathrm{Ev}_C(\mathcal F) = (R\Phi_f(\mathcal G))|_{\mathbb O_C\times_\Lambda \mathbb S}.$$

\subsection{Examples}

We provide two examples: one concerning tori, which is essentially the case of $\mathrm{GL}_1$, and another on $\mathrm{PGL}_2$, the `smallest' non-abelian case.

\subsubsection{Tori}

If $T$ is a split torus, then there is no non-zero nilpotent element in Lie$(\hat T)$, and hence $V_\lambda $ is a point and $H_\lambda = \hat T$. Moreover, $\hat T$ is a smooth group, and so each fibre of $\mathrm{Sing} \to \mathrm{Par}_G$ is just a point, i.e., $\mathrm{Sing}_\lambda =  [*/\hat T ] $. Our construction directly gives $\mathcal A_\varphi = \mathcal E_\varphi$, and hence our conjecture is compatible with Fargues' Conjecture \ref{Fargues-conjecture-detailed}.

\subsubsection{PGL$_2$}

Consider the parameter $\lambda = \nu^{-1/2}\oplus \nu^{-1/2}$ of $G = \mathrm{PGL}_2$. In this case, $V_\lambda = \mathbb A^1$ and $H_\lambda = \mathbb G_m$ acts by $(t,x)\mapsto t^2x$. Let $C_0$ (resp. $C_1$) be the zero (resp. non-zero) orbit. Both orbits are of Arthur type, namely:
$$\psi_0(w,a,b) = [2](b), \quad \psi_1(w,a,b) = [2](a),$$
where $[n]$ is the $n$-dimensional irreducible algebraic representation of $\mathrm{SL}_2(\Lambda)$, and $(m)$ denotes the $m$th monodromy.

There are three simple perverse sheaves on $[V_\lambda/H_\lambda]$, up to isomorphism:
$$\mathbf 1_{C_0}^\sharp:=\mathrm{IC}(C_0, \mathbf 1_{C_0}), \quad \mathbf 1_{C_1}^\sharp:=\mathrm{IC}(C_1, \mathbf 1_{C_1}), \quad \mathcal E_{C_1}^\sharp:=\mathrm{IC}(C_1, \mathcal E_{C_1}).$$
According to the results in \cite[Sec 12.3]{Cunningham-Voganish-begins}, they correspond to the following representations
$$(G_0,\mathbf  1_{G_0}), \quad (G_0, \mathrm{St}_{G_0}), \quad (G_1, \mathbf 1_{G_1}).$$
where $G_0 = G$ and $G_1 = G_{1/2}$ as described in Section \ref{subsection Bun-G}. To explain, by Kazhdan-Lusztig's theory \cite[Th 7.12]{Kazhdan-Lusztig}, the generic representation $\mathrm{St}_{G_0}$ should correspond to the open stratum with trivial monodromy. This correspondence differs from the one in (\ref{subsubsection Example: GL2}), which can be explained by applying a unipotent linear change on the semi-simplifications of the relevant sheaves (\emph{cf.} the geometric multiplicity matrix in \cite[Sec 12.2.2]{Cunningham-Voganish-begins}):
$$(\mathbf 1_{C_0}^\sharp)^{\mathrm{ss}} = (j_{C_0,!}\mathbf 1_{C_0})^{\mathrm{ss}},\quad 
(\mathbf 1_{C_1}^\sharp )^{\mathrm{ss}}= (j_{C_1,!}\mathbf 1_{C_1})^{\mathrm{ss}} - (j_{C_0,!}\mathbf 1_{C_0})^{\mathrm{ss}},$$
followed by a duality on orbits $C_0\leftrightarrow C_1$ \cite[Sec 12.2.1]{Cunningham-Voganish-begins}.

The regular conormal bundles are given by 
\begin{align*} 
\mathbb O_0 &= \{0\}\times( \mathbb A^1\smallsetminus\{0\} ), \quad 
\mathbb O_1 = (\mathbb A^1\smallsetminus\{0\})\times \{0\}, 
\end{align*}
with equivariant fundamental groups
\begin{align*} 
A_{\mathcal C_0} = 1\xleftarrow{}  A_{\mathbb O_0} & = \{\pm 1\} \xrightarrow{\mathrm{id}} A_{\mathcal C_1}^*  = \{\pm 1\} ,
\\
A_{\mathcal C_1} = \{\pm 1\} \xleftarrow{\mathrm{id}} 
 A_{\mathbb O_1}&= \{\pm 1\} \xrightarrow{} A_{\mathcal C_0}^*  = 1.
\end{align*}
Following the calculations in \cite[Sec 12]{Cunningham-Voganish-begins}, the vanishing cycles are given by
\begin{align*} 
\mathrm{Ev}_{C_0}&:\mathrm{IC}(\mathbf 1_{C_0})\mapsto \mathbf 1_{\mathbb O_0}, \quad \mathrm{IC}(\mathbf 1_{C_1}) \mapsto 0, \quad \mathrm{IC}(\mathcal E_{C_1})\mapsto \mathrm{IC}(\mathcal E_{\mathbb O_0})[1],
\\
\mathrm{Ev}_{C_1}&:\mathrm{IC}(\mathbf 1_{C_0})\mapsto 0, \quad \mathrm{IC}(\mathbf 1_{C_1}) \mapsto \mathbf 1_{\mathbb O_1}, \quad \mathrm{IC}(\mathcal E_{C_1})\mapsto \mathrm{IC}(\mathcal E_{\mathbb O_1}),
\end{align*}
and hence the Arthur sheaves are given by 
\begin{align*} 
\mathcal A_{C_0} = \mathrm{IC}(\mathbf 1_{C_0}) \oplus \mathrm{IC}(\mathcal E_{C_1}),\quad  
\mathcal A_{C_1} = \mathrm{IC}(\mathbf 1_{C_1}) \oplus \mathrm{IC}(\mathcal E_{C_1}).
\end{align*}
Assuming Conjecture (\ref{Fargues-conjecture-individual}) for individual representations, we obtain the Arthur packets via the action of the Arthur sheaves on the Whittaker sheaves: 
$$\Pi_{\psi_0} = \{\mathbf 1_{G_0}, \mathbf 1_{G_1}\}, \quad \Pi_{\psi_1} = \{\mathrm{St}_{G_0}, \mathbf 1_{G_1}\}.$$

\addcontentsline{toc}{section}{References} 
\newcommand{\etalchar}[1]{$^{#1}$}

\end{document}